\documentclass[letterpaper,11pt]{amsart}

\usepackage[margin=1.2in]{geometry}
\usepackage{amsmath,amsthm,amssymb}
\usepackage{xspace,xcolor}
\usepackage[breaklinks,colorlinks,citecolor=teal,linkcolor=teal,urlcolor=teal,pagebackref,hyperindex]{hyperref}
\usepackage[alphabetic]{amsrefs}
\usepackage[all]{xy}
\usepackage{color}

\theoremstyle{plain}
\newtheorem{thm}{Theorem}[section]

\newtheorem{lem}[thm]{Lemma}
\newtheorem{prop}[thm]{Proposition}
\newtheorem{cor}[thm]{Corollary}

\theoremstyle{definition}
\newtheorem{defi}[thm]{Definition}

\newtheorem{eg}[thm]{Example}

\newtheorem{question}[thm]{Question}

\theoremstyle{remark}
\newtheorem{rmk}[thm]{Remark}

\def\Z{{\mathbf Z}}
\def\Q{{\mathbf Q}}

\def\C{{\mathbf C}}

\def\Qalg{\overline{\mathbf Q}}

\def\A{{\mathbf A}}

\def\cJ{\mathcal{J}}

\def\cM{\mathcal{M}}

\def\cO{\mathcal{O}}

\def\J{\mathcal{J}}

\def\fra{\mathfrak{a}}
\def\frb{\mathfrak{b}}
\def\frc{\mathfrak{c}}
\def\frm{\mathfrak{m}}

\def\.{\cdot}
\def\^{\widehat}

\def\({\left(}
\def\){\right)}

\newcommand{\llbracket}{[\negthinspace[}
\newcommand{\rrbracket}{]\negthinspace]}

\renewcommand{\and}{ \ \ \text{ and } \ \ }

\DeclareMathOperator{\Spec} {Spec}

\DeclareMathOperator{\ord} {ord}

\DeclareMathOperator{\lct} {lct}

\DeclareMathOperator{\moi}{moi}
\DeclareMathOperator{\Vol}{vol}
\def\oi{{_K\rm{oi}}}

\begin{document}

\author[R.~Cluckers]
{Raf Cluckers}
\address{Univ.~Lille,
CNRS, UMR 8524 - Laboratoire Paul Painlevé, F-59000 Lille, France, and
KU Leuven, Department of Mathematics, B-3001 Leu\-ven, Bel\-gium}
\email{Raf.Cluckers@univ-lille.fr}
\urladdr{http://rcluckers.perso.math.cnrs.fr/}

\author[J.~Koll\'{a}r]{J\'{a}nos Koll\'{a}r}
\address{Department of Mathematics, Princeton University, Fine Hall, Washington Road,
Princeton, NJ 08544}
\email{kollar@math.princeton.edu}

\author[M.~Musta{\c{t}}{\u{a}}]{Mircea Musta{\c{t}}{\u{a}}}
\address{Department of Mathematics, University of Michigan, 530 Church Street, Ann Arbor, MI 48109, USA}
\email{mmustata@umich.edu}

\thanks{R.C. was partially supported by the European Research Council under the European Community's Seventh Framework Programme (FP7/2007-2013) with ERC Grant Agreement nr. 615722
MOTMELSUM, by the Labex CEMPI  (ANR-11-LABX-0007-01), and by KU Leuven IF C14/17/083. J. K. was partially supported  by  the NSF under grant number 
DMS-1901855.  M.M. was partially supported by NSF grants DMS-2001132 and DMS-1952399.}

\subjclass[2010]{14B05, 14E18, 14J17, 11L07}

\begin{abstract}
We show that if $f$ is a nonzero, noninvertible function on a smooth complex variety $X$ and $J_f$ is the Jacobian
ideal of $f$, then $\lct(f,J_f^2)>1$ if and only if the hypersurface defined by $f$ has rational singularities. Moreover,
if it does not have rational singularities, then $\lct(f,J_f^2)=\lct(f)$. We give two proofs, one relying on arc spaces and one that
goes through the inequality
$\widetilde{\alpha}(f)\geq\lct(f,J_f^2)$, where $\widetilde{\alpha}(f)$ is the minimal exponent of $f$.
In the case of a polynomial over $\overline{\Q}$, we also prove an analogue of this latter inequality,
with $\widetilde{\alpha}(f)$  replaced by the motivic oscillation index $\moi(f)$. We also show a part of Igusa's strong monodromy conjecture, for poles larger than $-\lct(f,J_f^2)$.  We end with a discussion of \emph{lct-maximal} ideals: these are ideals $\fra$ with the property that $\lct(\fra)<\lct(\frb)$ for every $\frb$ with $\fra\subsetneq\frb$.
\end{abstract}

\title{The log canonical threshold and rational singularities}

\maketitle

\section{Introduction}

Given a smooth complex algebraic variety $X$ and a (nonempty) hypersurface $H$ in $X$ defined by $f\in\cO_X(X)$, the log canonical
threshold $\lct(f)$ measures how far the pair $(X,H)$ is from having log canonical singularities. In particular, we always have
$\lct(f)\leq 1$, with equality if and only if the pair $(X,H)$ is log canonical. The log canonical threshold $\lct({\mathfrak a})$ can be defined
more generally for every nonzero coherent ideal ${\mathfrak a}$ of $\cO_X$ (with the convention that $\lct(\cO_X)=\infty$) and it can happen
that $\lct({\mathfrak a})>1$.
For an introduction to singularities of pairs in our setting, we refer the reader to \cite[Chapter~9]{Lazarsfeld}.

In this note we show that one can use the log canonical
threshold of an ideal associated to $f$ in order to refine $\lct(f)$, so that we detect when the hypersurface $H$ has rational singularities.
Namely, we consider the Jacobian ideal $J_f$ of $f$ and the log canonical threshold $\lct(f,J_f^2)$ of the ideal $(f,J_f^2):=(f)+J_f^2$.
We show that $\lct(f,J_f^2)>1$ if and only if the hypersurface $H$ has rational singularities. More precisely, we have the following:

\begin{thm}\label{thmA_intro}
For every smooth complex algebraic variety $X$, and every nonzero, noninvertible $f\in \cO_X(X)$ defining the hypersurface
$H$ in $X$, the following hold:
\begin{enumerate}
\item[i)] If $H$ does not have rational singularities, then
$$\lct(f,J_f^2)=\lct(f).$$
In particular, we have $\lct(f,J_f^2)\leq 1$.
\item[ii)] If $H$ has rational singularities, then $\lct(f,J_f^2)>1$.
\end{enumerate}
\end{thm}

The same result holds in the complex analytic setting, that is, if $X$ is a complex manifold and $f$ is a holomorphic function (see Remark~\ref{analytic}).
We note that the interesting assertion in Theorem~\ref{thmA_intro} is the one in i),
as the assertion in ii) follows easily from known properties of rational singularities.
We give two stronger versions of the result in Theorem~\ref{thmA_intro}.
The first of these is formulated in terms of another invariant of singularities, Saito's
\emph{minimal exponent} $\widetilde{\alpha}(f)$.
When $H$ has isolated singularities, this has been also known as the \emph{complex singularity index} of $H$. In general,
it is defined as the negative of the largest root of $b_f(s)/(s+1)$, where $b_f(s)$ is the Bernstein-Sato polynomial of $f$ (with the convention
that if $f$ defines a smooth hypersurface, in which case $b_f(s)=s+1$, then $\widetilde{\alpha}(f)=\infty$).
It is a result of Lichtin and
Koll\'{a}r (see \cite[Theorem~1.6]{Kollar}) that $\lct(f)=\min\big\{\widetilde{\alpha}(f),1\big\}$ and it was shown by Saito (see \cite[Theorem~0.4]{Saito-B})
that $\widetilde{\alpha}(f)>1$ if and only if $H$ has rational singularities. We thus see that $\widetilde{\alpha}(f)$ behaves like the invariant $\lct(f,J_f^2)$
in the previous theorem.
We prove the following general inequality between these two invariants:

\begin{thm}\label{thmB_intro}
For every smooth complex algebraic variety $X$, and every nonzero, noninvertible $f\in \cO_X(X)$, we have
$$\widetilde{\alpha}(f)\geq\lct(f,J_f^2).$$
\end{thm}

Another improvement on Theorem~\ref{thmA_intro} is formulated in terms of multiplier and adjoint ideals. These are important objects in birational geometry,
for the definition and basic applications we refer to \cite[Chapter 9]{Lazarsfeld}.

\begin{thm}\label{corC_intro}
For every smooth complex algebraic variety $X$, and every nonzero, noninvertible $f\in \cO_X(X)$, we have
$$\cJ(X,f^{\lambda})=\cJ\big(X, (f,J_f^2)^{\lambda}\big)\quad\text{for all}\quad \lambda<1.$$
Moreover, if $f$ defines a reduced hypersurface and ${\rm adj}(f)$ is the adjoint ideal of $f$, then
$${\rm adj}(f)=\cJ\big(X, (f,J_f^2)\big).$$
\end{thm}

We give two approaches to Theorem~\ref{thmA_intro}: one that proves
Theorem~\ref{thmB_intro} and one that proves Theorem~\ref{corC_intro}.
The first approach relies on showing that
given $f$, a point $P$ in the zero-locus of $f$, and $h\in J_f^2$,
if $g=f+ch$, where $c\in {\mathbf C}$ avoids at most $\dim(X)$ values (in fact, $c$ can be arbitrary if ${\rm mult}_P(f)\geq 3$),
then there is an automorphism of the completion of $\cO_{X,P}$ that maps $f$ to $g$.
We also show that if
${\mathfrak a}=(f_1,\ldots,f_r)$ is an ideal and $g=\sum_{i=1}^r\lambda_if_i$, with $\lambda_1,\ldots,\lambda_r\in {\mathbf C}$ general, then
$\widetilde{\alpha}(g)\geq \lct({\mathfrak a})$. By combining these two results, we easily deduce the inequality in Theorem~\ref{thmB_intro},
which in turn implies the first assertion in Theorem~\ref{thmA_intro}.

The second approach to Theorem~\ref{thmA_intro}  uses arc spaces and also gives an assertion of independent interest
 concerning divisorial valuations.  
 Recall that a \emph{divisorial valuation} is a valuation of the form $v=q\cdot {\rm ord}_E$,
where $E$ is a divisor on a normal variety that has a birational morphism to $X$ and $q$ is a positive integer; we denote by $A_X(v)$
the log discrepancy of $v$ and by $c_X(v)$ the center of $v$ on $X$ (for definitions, see Section~\ref{first_proof}).

\begin{thm}\label{thm3_intro}
Let $X$ be a smooth algebraic variety over an algebraically closed field
of arbitrary characteristic\footnote{We note that this theorem and the results in Section~\ref{formal_equivalence}
do not require that the ground field has characteristic $0$. All other results in this paper need this assumption, due to the fact 
that the basic results on invariants of singularities make use of the existence of resolution of singularities or of vanishing theorems.}, and $f\in\cO_X(X)$ a nonzero function.
If $v$ is a divisorial valuation on $X$ such that
$$0<v(J_f)<\tfrac{1}{2}v(f),$$
then there is a divisorial valuation $w$ on $X$ that satisfies the following conditions:
\begin{enumerate}
\item[i)] $w(g)\leq v(g)$ for every $g\in\cO_X(X)$,
\item[ii)] $w(f)\geq v(f)-1$,
\item[iii)] $A_X(w)\leq A_X(v)-1$, and
\item[iv)] $c_X(w)=c_X(v)$.
\end{enumerate}
\end{thm}

In characteristic $0$,
this result then implies the assertions in Theorem~\ref{corC_intro}. It also gives
a partial generalization of Theorem~\ref{thmA_intro} to arbitrary coherent nonzero ideals. For such an ideal $\fra$,
we denote by $D(\fra)$ the ideal that on affine open subsets is the sum of the Jacobian ideals of the sections of $\fra$ (note that if $\fra=(f)$, then
$D(\fra)=(f)+J_f$). With this notation, we have

\begin{cor}\label{corD_intro}
For every smooth complex algebraic variety $X$,
if $\fra$ is a coherent nonzero ideal on $X$ with $\lct(\fra)<1$, then $\lct\big(\fra+D(\fra)^2\big)=\lct(\fra)$.
\end{cor}

An amusing consequence of Theorem~\ref{thmB_intro} is the following inequality
involving the minimal exponent and the Milnor number of an isolated singularity:

\begin{cor}\label{cor_thmB_intro}
If $X$ is a smooth $n$-dimensional complex algebraic variety, $f\in\cO_X(X)$ is nonzero, and $P$ is an isolated singular point of the zero-locus of $f$,
with Milnor number $\mu_P(f)$,
then
$$\widetilde{\alpha}_P(f)^n\cdot\mu_P(f)\geq\frac{n^n}{2^n}.$$
\end{cor}

Recall that if $f$ has an isolated singularity at $P$, then the Milnor number
$\mu_P(f)$ is equal to ${\rm dim}_{{\mathbf C}}(\cO_{X,P}/J_f)$. Some examples showing
that the inequality in Corollary~\ref{cor_thmB_intro} is sharp are given in Remark~\ref{rmk_cor_thmB_intro}.

There is another invariant that behaves like $\widetilde{\alpha}_f$ and $\lct(f,J_f^2)$, namely the \emph{motivic oscillation index}
$\moi(f)$ studied in \cite{CMN}. This is defined for polynomials $f\in \overline{\Q}[x_1,\ldots,x_n]$ and
it was shown in
\cite[Proposition~3.10]{CMN} that $\lct(f)=\min\{\moi(f),1\}$ and $\moi(f)>1$ if and only if the hypersurface defined by $f$ in $\A_{\overline{\Q}}^n$
has rational singularities.

In fact, a more refined version $\moi_Z(f)$ of the motivic oscilation index also involves a closed subscheme $Z$ of $\A^n_{\overline{\Q}}$
(the one we referred to in the previous paragraph corresponds to the case when $Z$ is the hypersurface defined by $f$).
We recall the precise definition of $\moi_Z(f)$ in Section~\ref{moi}. We only mention now that if $f\in\Z[x_1,\ldots,x_n]$ and
$Z=\A^n_{\overline{\Q}}$, then $\moi_Z(f)$ relates to finite exponential sums over integers modulo $p^m$ (for primes $p$ and integers $m>0$)
of the form
$$
E(p^m) := \frac{1}{p^{mn}}\sum_{x\in(\Z/p^m\Z)^n} \exp \left(2\pi i \frac{f(x)}{p^m}\right)
$$
and to certain limit values of all possible $\sigma\geq 0$ such that
$$
|E(p^m)|  \ll p^{-m\sigma },
$$
with an implicit constant independent from $m$
(these limits are taken carefully, using in fact finite field extensions and large primes $p$, as will be described in Section~\ref{moi}).
We have the following inequality between $\moi(f)$ and $\lct(f,J_f^2)$:

\begin{thm}\label{thm_moi1}
For every nonconstant $f\in{\overline{\Q}}[x_1,\ldots,x_n]$, we have
\begin{equation}\label{eq_thm_moi1}
\moi(f)\geq\lct(f,J_f^2).
\end{equation}
Moreover, if $\lct(f,J_f^2)\leq 1$, then we have equality in (\ref{eq_thm_moi1}).
\end{thm}

Note that by Theorem~\ref{thmA_intro}, we have $\lct(f,J_f^2)\leq 1$ if and only if the hypersurface defined by $f$
does not have rational singularities. In fact, one can give a third proof of part i) of Theorem ~\ref{thmA_intro} by combining
Theorem~\ref{thm_moi1} and the mentioned assertions in \cite[Proposition~3.10]{CMN}; we leave the details to the reader.

We will prove a more general version of the above theorem, allowing for a subset
 $Z$ of the zero-locus of $f$ (see Theorem~\ref{thm_moi2} below).
A related intriguing question is whether we always have  $\widetilde{\alpha}_f=\moi(f)$. However, investigating this
seems to require new ideas, related to Igusa's strong monodromy conjecture. In Corollary \ref{cor:monodromy}, we prove a part of Igusa's strong monodromy conjecture, for poles with real part larger than $-\lct(f,J_f^2)$.

We end the paper by discussing a notion motivated by our results, that of \emph{lct-maximal} ideals. These are proper nonzero ideals $\fra$ such that for every
$\frb$ with $\fra\subsetneq\frb$, we have $\lct(\fra)<\lct(\frb)$. For example, if $H$ is a hypersurface in $X$ such that the pair $(X,H)$ is log canonical, then
it follows from Theorem~\ref{thmA_intro} that
the ideal
$\cO_X(-H)$ is lct-maximal if and only if $H$ has rational singularities.
 In general, we characterize when an ideal $\fra$ is lct-maximal in terms
of the divisorial valuations that compute $\lct(\fra)$ (see Proposition~\ref{prop_char1}). Using this, we obtain the following result:

\begin{thm}\label{thm_char2}
Let $\fra$ be a proper nonzero coherent ideal of $\cO_X$, where $X$ is a smooth complex algebraic variety.
\begin{enumerate}
\item[i)] If $\fra$ is lct-maximal (and not necessarily radical), then its zero-locus, with reduced structure, has rational singularities.
\item[ii)] If $\fra$ is a radical ideal, defining an irreducible, locally complete intersection subscheme $W$, with rational singularities,
then $\fra$ is lct-maximal.
\end{enumerate}
\end{thm}

The paper is organized as follows. In Section~\ref{formal_equivalence} we treat the formal equivalence of $f$ and that of $f+h$, for suitable
$h\in J_f^2$. In  Section~\ref{first_proof} we prove Theorem~\ref{thmB_intro} and deduce Corollary~\ref{cor_thmB_intro} and
Theorem~\ref{thmA_intro}. In Section~\ref{second_proof}, after reviewing some basic facts about the connection between valuations and contact loci
in arc spaces, we prove Theorem~\ref{thm3_intro}, deduce Theorem~\ref{corC_intro} and Corollary~\ref{corD_intro}, and obtain a second proof of Theorem~\ref{thmA_intro}.
In Section~\ref{examples}, we give two examples. We show that for generic determinantal hypersurfaces, the inequality in
Theorem~\ref{thmB_intro} is an equality, and we describe when this inequality is strict in the case of homogeneous diagonal hypersurfaces.
In Section~\ref{moi} we recall the definition of the motivic oscillation index and prove the general version of Theorem~\ref{thm_moi1}.
Finally, in Section~\ref{open_questions} we discuss lct-maximal ideals, prove Theorem~\ref{thm_char2}, and raise some open questions.

\subsection{Acknowledgments}
We would like to thank Nero Budur, Mattias Jonsson, Johannes Nicaise, Mihai P\u{a}un, and Uli Walther for useful discussions and to Nero Budur for his comments on an earlier version
of this paper.

\section{Adding terms in the square of the Jacobian ideal}\label{formal_equivalence}

Let $k$ be a field of arbitrary characteristic and $R=k\llbracket x_1,\ldots,x_n\rrbracket$, for a positive integer $n$.
We denote by $\frm$ the maximal ideal of $R$.
We say that two elements $f,g\in R$
are \emph{formally equivalent} (and write $f\sim g$) if there is an automorphism of $R$ that maps $f$ to $g$. Recall that for $f\in R$ nonzero, the \emph{multiplicity}
${\rm mult}(f)$ of $f$
is the largest $q$ such that $f\in\frm^q$. The \emph{Jacobian ideal} $J_f$ of $f$ is the ideal $\left(\frac{\partial f}{\partial x_1},\ldots,\frac{\partial f}{\partial x_n}\right)$.

We first recall the following version of Taylor's formula that's valid in any characteristic. We use the multi-index notation $v^{\alpha}=v_1^{\alpha_1}\cdots v_n^{\alpha_n}$
for $v=(v_1,\ldots,v_n)\in R^n$ and $\alpha=(\alpha_1,\ldots,\alpha_n)\in\Z_{\geq 0}^n$. We also denote by $D^{\alpha}$ the divided power differential operator $\prod_{k=1}^n\tfrac{\partial_{x_k}^{\alpha_k}}{\alpha_k!}$
given by
$$D^{\alpha}\left(\sum_{\beta\in\Z_{\geq 0}^n}c_{\beta}x^{\beta}\right)=\sum_{\beta\in\Z_{\geq 0}}{\beta_1\choose \alpha_1}\cdots {\beta_r\choose \alpha_r}c_{\beta}x^{\beta-\alpha},$$
with the convention that ${\beta_i\choose\alpha_i}=0$ if $\beta_i<\alpha_i$. 

\begin{lem}\label{Taylor_formula}
For every $f\in R$ and every $u=(u_1,\ldots,u_n)$, $v=(v_1,\ldots,v_n)\in\frm^{\oplus n}$, we have
\begin{equation}\label{eq_Taylor_formula}
f(u+v)=\sum_{\alpha\in{\mathbf Z}_{\geq 0}^n}D^{\alpha}f(u)v^{\alpha}.
\end{equation}
\end{lem}

\begin{proof}
Since both sides of (\ref{eq_Taylor_formula}) are additive and continuous with respect to the $\frm$-adic topology, it is enough to check the formula
for $f=x_1^{\beta_1}\cdots x_n^{\beta_n}$, when it is an immediate consequence of the binomial formula. 
\end{proof}

\begin{prop}\label{prop1_formal_equiv}
If $f\in R$ satisfies ${\rm mult}(f)=d\geq 3$, then for every $g\in J_f^2$, we have $f\sim f+g$.
\end{prop}

\begin{proof}
The argument follows closely the proof of Tougeron's finite determinacy theorem \cite[Section X.4]{Tougeron},
especially the version given in  \cite[Section 6.4]{AGZV}.
We divide it in three steps.

\noindent {\bf Step 1}. For every $f$ with ${\rm mult}(f)=d\geq 3$ and every $g\in J_f^2$, we have $J_f=J_{f+g}$.

Note first that $J_f\subseteq\frm^{d-1}\subseteq\frm^2$.
By assumption, we can write $g=\sum_{i,j}h_{i,j}\frac{\partial f}{\partial x_i}\frac{\partial f}{\partial x_j}$, for some $h_{i,j}\in R$.
This implies that for every $\ell$, we have $\frac{\partial g}{\partial x_{\ell}}\in \frm J_f$.
Since
$$\frac{\partial (f+g)}{\partial x_{\ell}}-\frac{\partial f}{\partial x_{\ell}}=\frac{\partial g}{\partial x_{\ell}} \in \frm J_f\quad\text{for every}\quad \ell,$$
it follows that $J_{f+g}\subseteq J_f$ and $J_f\subseteq J_{f+g}+\frm J_f$. We thus have $J_f=J_{f+g}$
by Nakayama's lemma.

\noindent {\bf Step 2}. Suppose that $a$ is a nonnegative integer and $g\in \frm^aJ_f^2$. Let us write
$$g=\sum_{i=1}^ng_i\frac{\partial f}{\partial x_i},\quad\text{with}\quad g_i\in\frm^aJ_f\,\,\text{for all}\,\,i.$$
In this case, it follows from Lemma~\ref{Taylor_formula}
 that
$$f(x_1+g_1,\ldots,x_n+g_n)-(f+g)$$
is a sum of terms that each lie in some $\frm^{d-i}(\frm^aJ_f)^i$, with $i\geq 2$ (with the convention that $\frm^j=R$ for $j<0$).
Since
$$\frm^{d-i}(\frm^aJ_f)^i\subseteq  \frm^{2a+1}J_f^2\quad\text{for every}\quad i\geq 2,$$
it follows that if we consider the automorphism $\varphi\colon R\to R$ given by $\varphi(x_i)=x_i+g_i$, then
$\varphi(f)-(f+g)\in \frm^{2a+1}J_f^2$. Note that $\varphi(x_i)-x_i=g_i\in \frm^aJ_f$.

\noindent {\bf Step 3}. Suppose now that $f$ and $g$ are as in the statement of the proposition.
We construct a sequence of autormorphisms $(\varphi_j)_{j\geq 1}$ of $R$ that satisfy
\begin{equation}\label{eq1_prop1_formal_equiv}
\varphi_j(x_i)-x_i\in \frm^{2^{j-1}-1}J_f\quad\text{for all}\quad 1\leq i\leq n\quad\text{and}
\end{equation}
\begin{equation}\label{eq2_prop1_formal_equiv}
\varphi_j\circ\ldots\circ\varphi_1(f)-(f+g)\in\frm^{2^j-1}J_f^2.
\end{equation}
We first apply Step 2 with $a=0$ to construct the automorphism $\varphi_1$ that satisfies (\ref{eq1_prop1_formal_equiv}) and
(\ref{eq2_prop1_formal_equiv}).
Suppose now that we have constructed $\varphi_j$ that satisfies (\ref{eq1_prop1_formal_equiv}) and
(\ref{eq2_prop1_formal_equiv}). Note first that
\begin{equation}\label{eq3_prop1_formal_equiv}
J_{\varphi_j\circ\ldots\circ\varphi_1(f)}=J_f.
\end{equation}
Indeed, it follows from Step 1 that $J_{f+g}=J_f$, hence (\ref{eq2_prop1_formal_equiv}) implies
$$J_{\varphi_j\circ\ldots\circ\varphi_1(f)}\subseteq J_f\subseteq J_{\varphi_j\circ\ldots\circ\varphi_1(f)}+\frm J_f.$$
Therefore (\ref{eq3_prop1_formal_equiv}) follows from Nakayama's lemma.
We then apply Step 2 with $a=2^j-1$ to construct $\varphi_{j+1}$ that satisfies (\ref{eq1_prop1_formal_equiv}) and (\ref{eq2_prop1_formal_equiv}).

Note now that (\ref{eq1_prop1_formal_equiv}) implies that for every $h\in R$, the sequence $\big(\varphi_j\circ\ldots\circ\varphi_1(h)\big)_{j\geq 1}$
converges in the $\frm$-adic topology to an element $\psi(h)\in R$. Since each $\varphi_j\circ\ldots\circ\varphi_1$ is a ring homomorphism, by passing to limit
 it follows that so is $\psi$. 
In fact, it is an automorphism (this is due to the fact 
that a ring homomorphism $\psi\colon R\to R$ is an automorphism if and only if $P_i:=\psi(x_i)\in\frm$ for all $i$ and ${\rm det}\left(\tfrac{\partial P_i}{\partial x_j}\right)\not\in\frm$,
and these conditions are preserved after passing to limit in the $\frm$-adic topology).
We then conclude using (\ref{eq2_prop1_formal_equiv}) that $\psi(f)=f+g$. This completes the proof of the proposition.
\end{proof}

We next turn to the case when ${\rm mult}(f)=2$. From now on we assume that ${\rm char}(k)\neq 2$.
Given any $f\in R$ with ${\rm mult}(f)\geq 2$, we write $f_2$ for the sum of the degree 2 terms in $f$ and $f_{\geq 3}=f-f_2$.
Suppose that $f_2\neq 0$ and let $r={\rm rank}(f_2)$. After applying an automorphism of $R$ given by a linear change of variables, we
may assume that $f_2=\sum_{i=1}^ra_ix_i^2$, for some $a_1,\ldots,a_r\in k^*$.
Recall now that if this is the case, then there is an automorphism
$\varphi$ of $R$ that fixes $x_{r+1},\ldots,x_n$, such that
\begin{equation}\label{eq_Morse1}
\varphi(f)=\sum_{i=1}^ra_ix_i^2+h(x_{r+1},\ldots,x_n),
\end{equation}
for some $h\in k\llbracket x_{r+1},\ldots,x_n\rrbracket$ with ${\rm mult}(h)\geq 3$
(this is the content of the Morse lemma, see \cite[page 6]{Milnor} and \cite[Section~6.2]{AGZV} for the precise
differentiable and analytic versions).
We can achieve this as follows. Clearly, arguing by induction on $r$,
it is enough to find $\varphi$ that fixes $x_2,\ldots,x_n$, such that $\varphi(f)_2=f_2$, and
$$
\varphi(f)=a_1x_1^2+h(x_2,\ldots,x_n)
$$
(we will refer to this as \emph{morsification} with respect to $x_1$).
Let us write
\begin{equation}\label{eq_Morse2}
f=x_1^2a(x_1,\ldots,x_n)+x_1b(x_2,\ldots,x_n)+h(x_2,\ldots,x_n),
\end{equation}
so that $a_1:=a(0,\ldots,0)\neq 0$ and ${\rm mult}(b)\geq 2$.
If $\varphi_1$ is the automorphism of $R$ given by
$\varphi_1(x_1)=x_1-\frac{b}{2a_1}$ and $\varphi_1(x_i)=x_i$ for $i\geq 2$, then
$\varphi_1(f)_2=f_2$ (since ${\rm mult}(b)\geq 2$)
and if we write as above
$$\varphi_1(f)=x_1^2a'(x_1,\ldots,x_n)+x_1b'(x_2,\ldots,x_n)+c'(x_2,\ldots,x_n),$$
then ${\rm mult}(b')>{\rm mult}(b)$. We can thus recursively construct a sequence of automorphisms
$(\varphi_i)_{i\geq 1}$ of $R$ and another such automorphism $\psi$ such that for every $w\in R$,
$\varphi_j\circ\ldots\circ\varphi_1(w)$ converges to $\psi(w)$ in the $\frm$-adic topology and after replacing $f$ by $\psi(f)$,
we may assume that in (\ref{eq_Morse2}) we have $b=0$. Since $a_1^{-1}a$ has constant term equal to $1$, it follows that
there is an invertible $p\in R$ such that $p^2=a_1^{-1}a$. If $\varphi$ is the automorphism of $R$ given by $\varphi(x_1)=x_1p$ and
$\varphi(x_i)=x_i$ for $i\geq 2$, it follows that $\varphi(f)_2=f_2$ and $\varphi(f)=a_1x_1^2+h(x_2,\ldots,x_n)$.

\begin{rmk}
Of course, if $k$ is algebraically closed, we conclude from (\ref{eq_Morse1}) that
$$f\sim \sum_{i=1}^rx_i^2+h(x_{r+1},\ldots,x_n).$$
\end{rmk}

\begin{rmk}\label{rmk_Morsification}
Suppose that 
$$f=\sum_{i=1}^ra_ix_i^2+f_{\geq 3}\in R,$$
where $a_i\neq 0$ for $1\leq i\leq r$, and let $\frb\subseteq k\llbracket x_{2},\ldots,x_n\rrbracket$ be an ideal 
such that $f\in \fra=(x_1R+\frb R)^2$. It follows from the description of the Morsification algorithm with respect to $x_1$
that the automorphism $\varphi$ of $R$ that we constructed (that fixes $x_2,\ldots,x_n$ and such that $\varphi(f)_2=f_2$ and 
$\varphi(f)=a_1x_1^2+q_2(x_2,\ldots,x_n)$) has the property that $\varphi(\fra)\subseteq\fra$. Therefore we have $q_2\in\frb^2$. 

Iterating this observation $r$ times, we see that if $\frc\subseteq k\llbracket x_{r+1},\ldots,x_n\rrbracket$ is an ideal such that
$f\in \big((x_1,\ldots,x_r)+\frc R\big)^2$, then by successively doing Morsification with respect to $x_1,\ldots,x_r$, we get an automorphism $\psi$
of $R$ that fixes $x_{r+1},\ldots,x_n$, and such that 
$$\psi(f)=\sum_{i=1}^ra_ix_i^2+q_{r+1}(x_{r+1},\ldots,x_n),\quad\text{with}\quad q_{r+1}\in \frc^2.$$
\end{rmk}

\begin{prop}\label{prop2_formal_equiv}
Suppose that ${\rm char}(k)\neq 2$.
If
$$f= \sum_{i=1}^ra_ix_i^2+h(x_{r+1},\ldots,x_n),$$
with $a_1,\ldots,a_r\in k^*$ and ${\rm mult}(h)\geq 3$, then for every $g\in J_f^2$ such that ${\rm rank}(f_2+g_2)=r$, there are
$c_1,\ldots,c_r\in k^*$ such that
$$f+g\sim\sum_{i=1}^rc_ix_i^2+h(x_{r+1},\ldots,x_n).$$
In particular, if $k$ is algebraically closed, then we have $f+g\sim f$.
\end{prop}

\begin{proof}
Note that we have $J_f=(x_1,\ldots,x_r)+J_hR$, where $J_h$ is the Jacobian ideal of $h$ in $k\llbracket x_{r+1},\ldots,x_n\rrbracket$.
Since $g\in J_f^2$, we have $f_2+g_2\in k\llbracket x_1,\ldots,x_r\rrbracket$. 
By hypothesis, we have  ${\rm rank}(f_2+g_2)=r$, so that there is an automorphism $\varphi$ of $R$ given by a linear change of coordinates in $x_1,\ldots,x_r$ such that
$$\varphi(f+g)_2=c_1x_1^2+\ldots+c_rx_r^2,$$
for some $c_1,\ldots,c_r\in k^*$. 
Since
$$\varphi(f+g)-h\in \big((x_1,\ldots,x_r)+J_hR\big)^2,$$
applying Remark~\ref{rmk_Morsification} for $\varphi(f+g)-h$, we see that there is an automorphism $\psi$ of $R$
that fixes $x_{r+1},\ldots,x_n$ (hence $\psi(h)=h$) and such that 
$$\psi\big(\varphi(f+g)\big)=\sum_{i=1}^rc_ix_i^2+q+h,$$
for some $q\in J_h^2\subseteq k\llbracket x_{r+1},\ldots,x_n\rrbracket$. Since $q+h\sim h$ in $k\llbracket x_{r+1},\ldots,x_n\rrbracket$
by Proposition~\ref{prop1_formal_equiv}, we conclude that $f+g\sim \sum_{i=1}^rc_ix_i^2+h$.
\end{proof}

\section{The inequality between $\widetilde{\alpha}(f)$ and $\lct(f,J_f^2)$}\label{first_proof}

Let us begin by recalling some terminology and notation regarding valuations that will be used from now on.
Let $X$ be a smooth (irreducible) $n$-dimensional algebraic variety over an algebraically closed field $k$
(we don't assume ${\rm char}(k)=0$ to begin with, since we will use the definitions below also in the next section, when the ground field will be allowed to have 
positive characteristic). 
All ideals of $\cO_X$ we consider are coherent.
For basic facts about log canonical thresholds and multiplier ideals (over a field of characteristic $0$) we refer to \cite[Chapter~9]{Lazarsfeld}.

A \emph{divisorial valuation} on $X$
is a valuation of the function field $k(X)$ of $X$ of the form $q\cdot {\rm ord}_E$, where $q$ is a positive integer
and $E$ is a prime divisor on a normal variety $Y$ that has a birational morphism $g\colon Y\to X$
(here ${\rm ord}_E$ is the discrete valuation associated to $E$, with corresponding DVR $\cO_{Y,E}$, having
fraction field $k(Y)=k(X)$). After replacing $Y$ by a suitable open subset, we may always assume that
$Y$ is smooth and $E$ is a smooth prime divisor on $Y$. The \emph{center} $c_X(v)$ of $v=q\cdot {\rm ord}_E$ on $X$
is the closure of $g(E)$ (which is independent of the model $Y$).
For an ideal $\fra$ of $\cO_X$, we denote by $v(\fra)$ the minimum of $v(g)$, where $g$ runs over the sections
of $\fra$ on an open subset that meets $c_X(v)$.

The \emph{log discrepancy} of $q\cdot {\rm ord}_E$
is the positive integer
$$A_X(q\cdot {\rm ord}_E)=q\cdot \big({\rm ord}_E(K_{Y/X})+1\big),$$
where $K_{Y/X}$ is the effective divisor on $Y$ locally defined by the determinant of the Jacobian matrix of $g$.

Recall that for a proper nonzero ideal $\fra$ of $X$, we have
$$\lct(\fra)=\inf_v\frac{A_X(v)}{v(\fra)},$$
where the minimum is over all divisorial valuations $v$ on $X$.
If ${\rm char}(k)=0$, then it follows from the existence of log resolutions that the infimum in the formula is a minimum.
We say that $v$ \emph{computes} $\lct(\fra)$ if
$v$ achieves this minimum.

From now on we assume $k=\C$ (one could also just assume ${\rm char}(k)=0$).
We first give a proposition
concerning the minimal exponent of a general linear combination of the generators of an ideal.

\begin{prop}\label{gen_combination}
If $f_1,\ldots,f_r\in\cO_X(X)$ generate the proper nonzero ideal ${\mathfrak a}$ of $\cO_X$ and
$f=\sum_{i=1}^r\lambda_if_i$, with $\lambda_1,\ldots,\lambda_r\in\C$ general, then we have
$$\widetilde{\alpha}(f)\geq\lct({\mathfrak a}).$$
\end{prop}

\begin{proof}
If the zero-locus $Z$ of ${\mathfrak a}$ has codimension 1 in $X$, then ${\rm lct}({\mathfrak a})\leq 1$ and in this case we have
$$\widetilde{\alpha}(f)\geq {\rm lct}(f)={\rm lct}({\mathfrak a}).$$
The equality follows from the fact that for every $t\in (0,1)$, we have the equality of multiplier ideals
$${\mathcal J}(f^t)={\mathcal J}({\mathfrak a}^t)$$
(see \cite[Proposition~9.2.26]{Lazarsfeld}). We thus may and will assume that ${\rm codim}_X(Z)\geq 2$.

The argument then proceeds as in ${\mathit loc.\,cit.}$ Let $\pi\colon Y\to X$ be a log resolution of $(X, {\mathfrak a})$ that is an isomorphism
over $X\smallsetminus Z$. By construction, if we put ${\mathfrak a}\cdot\cO_Y=\cO_Y(-E)$, then $E$ is a simple normal crossing divisor
such that if we write $E=\sum_{i=1}^Na_iE_i$, then every $E_i$ is a $\pi$-exceptional divisor. Since $\lambda_1,\ldots,\lambda_r$ are general,
it follows that if $D$ is the divisor defined by $f$, then $\pi^*(D)=F+E$, where $F$ is a smooth divisor, with no common components with $E$, and
having simple normal crossings with $E$. In particular, $D$ is a reduced divisor and $\pi$ is a log resolution of $(X,D)$ such that the strict transform of $D$ is smooth. We thus deduce
using  \cite[Corollary~D]{MP} (or \cite[Corollary~1.5]{DM})
that if $K_{Y/X}=\sum_{i=1}^Nk_iE_i$, then
$$
\widetilde{\alpha}(f)\geq\min_{i=1}^N\frac{k_i+1}{a_i}=\lct({\mathfrak a}),$$
giving the inequality in the proposition.
\end{proof}

In what follows we will also make use of a local version of the minimal exponent. Recall that if $f\in\cO_X(X)$ is nonzero and $P\in X$ is such that
$f(P)=0$, then $\max\{\widetilde{\alpha}(f\vert_U)\mid U\ni P\}$, where $U$ varies over the open neighborhoods of $P$,
 is achieved for all small enough $U$. This maximum is denoted $\widetilde{\alpha}_P(f)$ and we have $\widetilde{\alpha}(f)=\min
\{ \widetilde{\alpha}_{P}(f)\mid P\in X\}$. Equivalently, $\widetilde{\alpha}_P(f)$ is the negative of the largest root of $b_{f,P}(s)/(s+1)$, where $b_{f,P}(s)$
is the Bernstein-Sato polynomial of the image of $f$ in $\cO_{X,P}$. Indeed, it follows easily from the definition that $b_{f,P}(s)$ is the greatest common divisor
of the $b_{f\vert_U}(s)$, where $U$ varies over the open neighborhoods of $P$, and we have $b_{f,P}(s)=b_{f\vert_U}(s)$ if $U$ is small enough.

\begin{rmk}\label{rmk_gen_combination}
With the same notation as in Proposition~\ref{gen_combination}, for every $P\in X$, if $\lambda_1,\ldots,\lambda_r\in \C$ are general, then we
have
$$\widetilde{\alpha}_{P}(f)\geq\lct_P({\mathfrak a}).$$
Indeed, if we choose an open neighborhood $U$ of $P$ such that $\lct_P({\mathfrak a})=\lct({\mathfrak a}\vert_U)$, then applying the proposition
for ${\mathfrak a}\vert_U$, we obtain
$$\widetilde{\alpha}_{P}(f)\geq \widetilde{\alpha}(f\vert_U)\geq \lct({\mathfrak a}\vert_U)=\lct_P({\mathfrak a}).$$
\end{rmk}

Given $f\in\cO_X(X)$, we denote by $J_f$ the \emph{Jacobian ideal} of $f$. Recall that if $U$ is an open subset of $X$
such that $x_1,\ldots,x_n$ are algebraic coordinates on $U$ (that is, $dx_1,\ldots,dx_n$ give a trivialization of
$\Omega_U$), then $J_f$ is generated on $U$ by $\frac{\partial f}{\partial x_1},\ldots,
\frac{\partial f}{\partial x_n}$ (the fact that this does not depend on the system of coordinates is straightforward to check).
If $P$ is a (closed) point on $X$, then we have an isomorphism $\widehat{\cO_{X,P}}\simeq \C\llbracket x_1,\ldots,x_n\rrbracket$
and the above definition of $J_f$ is compatible, via passing to completion, with the one we gave in Section~\ref{formal_equivalence}.
Note that if $f(P)=0$, then $P$ is an isolated point in the zero-locus of $J_f$ if and only if
$f$ has an isolated singular point at $P$. In this case, we consider the \emph{Milnor number}
$\mu_P(f)=\dim_{\C}(\cO_{X,P}/J_f)$.

We next give a lemma showing that the minimal exponent does not change under formal equivalence.

\begin{lem}\label{invar_formal_equivalence}
Given nonzero $f,g\in\cO_X(X)$ and $P\in X$ such that $f(P)=0=g(P)$, if the images of $f$ and $g$ in $\widehat{\cO_{X,P}}$ are formally equivalent, then
$\widetilde{\alpha}_P(f)=\widetilde{\alpha}_P(g)$.
\end{lem}

\begin{proof}
Recall that the existence of the Bernstein-Sato polynomial has been proved in \cite{Bjork} for rings of formal power series; in particular, this applies to
$\widehat{\cO_{X,P}}$. Moreover, if $b_{f,P}(s)$ and $b_{\widehat{f},P}(s)$ are the Bernstein-Sato polynomials of the images of $f$ in $\cO_{X,P}$ and
$\widehat{\cO_{X,P}}$, respectively, then it is easy to deduce, using the definition of Bernstein-Sato polynomials and the fact that the morphism
$\cO_{X,P}\to\widehat{\cO_{X,P}}$ is faithfully flat, that $b_{f,P}(s)=b_{\widehat{f},P}(s)$. Since $f$ and $g$ are formally equivalent, we have
$b_{\widehat{f},P}(s)=b_{\widehat{g},P}(s)$ and thus $b_{f,P}(s)=b_{g,P}(s)$. Dividing by $(s+1)$ and taking the negative of the largest roots gives the assertion in the proposition.
\end{proof}

We can now prove the inequality between $\widetilde{\alpha}(f)$ and $\lct(f,J_f^2)$:

\begin{proof}[Proof of Theorem~\ref{thmB_intro}]
It is enough to show that for every $P$ with $f(P)=0$, we have
$\widetilde{\alpha}_P(f)\geq {\rm lct}_P(f,J_f^2)$. We may assume that $f$ is singular at $P$ since otherwise the assertion is trivial.
Furthermore, we may assume that we have algebraic coordinates $x_1,\ldots,x_n$ on $X$
and let $g=c_0 f+\sum_{i,j=1}^nc_{i,j}\frac{\partial f}{\partial x_i}\frac{\partial f}{\partial x_j}$, with $c_0, c_{i,j}\in\C$ general.
In particular, we have $c_0\neq 0$ and $c_0^{-1}g-f\in J_f^2$, hence it follows from Propositions~\ref{prop1_formal_equiv} and \ref{prop2_formal_equiv}
that the images of $c_0^{-1}g$ and $f$ in $\widehat{\cO_{X,P}}$ are formally equivalent. We thus get
$\widetilde{\alpha}_P(f)=\widetilde{\alpha}_P(c_0^{-1}g)=\widetilde{\alpha}_P(g)$ by Lemma~\ref{invar_formal_equivalence}.
Since $\widetilde{\alpha}_P(g)\geq {\rm lct}_P(f,J_f^2)$ by Remark~\ref{rmk_gen_combination}, we conclude that $\widetilde{\alpha}_P(f)\geq {\rm lct}_P(f,J_f^2)$.
\end{proof}

We next deduce the inequality involving the minimal exponent and the Milnor number in the case of isolated singularities:

\begin{proof}[Proof of Corollary~\ref{cor_thmB_intro}]
If $J$ is a coherent ideal on $X$ and $P$ is an isolated point of the zero-locus of $J$, then
$$e(J,\cO_{X,P})\cdot {\rm lct}_P(J)^n\geq n^n,$$
where $e(J,\cO_{X,P})$ is the Hilbert-Samuel multiplicity of $\cO_{X,P}$ with respect to the ideal $J\cdot\cO_{X,P}$
(see \cite[Theorem~0.1]{dFEM}). We apply this with $J=J_f$. Since $J_f\cdot\cO_{X,P}\subseteq \cO_{X,P}$ is a $0$-dimensional ideal
generated by $n$ elements, it follows that it is generated by a regular sequence, hence $e(J_f,\cO_{X,P})=\ell(\cO_{X,P}/J_f)=\mu_P(f)$
(see \cite[Theorem~14.11]{Matsumura}).

On the other hand, it follows from Theorem~\ref{thmB_intro} that
$$\widetilde{\alpha}_P(f)\geq {\rm lct}_P(f,J_f^2)\geq {\rm lct}_P(J_f^2)=\frac{{\rm lct}_P(J_f)}{2}.$$
By combining these facts we obtain
$$n^n\leq \mu_P(f)\cdot {\rm lct}_P(J_f)^n\leq \mu_P(f)\cdot 2^n\cdot \widetilde{\alpha}_P(f)^n,$$
which gives the assertion in the corollary.
\end{proof}

\begin{rmk}\label{rmk_cor_thmB_intro}
If $f\in {\mathbf C}[x_1,\ldots,x_n]$ is homogeneous of degree $d\geq 2$, with an isolated singularity, then $\mu_P(f)=(d-1)^n$ and
$\widetilde{\alpha}(f)=\frac{n}{d}$. In particular, we see that if $d=2$, then the inequality in Corollary~\ref{cor_thmB_intro} becomes an equality.
It would be interesting to find stronger inequalities that are satisfied when we avoid some special cases (for example, when we assume
that ${\rm mult}_P(f)\geq 3$). In this direction, note that it is shown in \cite[Theorem~1.6]{AS} that if $n=2$
(in which case $\widetilde{\alpha}_P(f)={\rm lct}_P(f)$) and $f$ is analytically
irreducible at $P$, with a singularity whose value semigroup is different from $\langle 2,3\rangle$ and $\langle 2,5\rangle$, then
$${\rm lct}_P(f)^2\cdot\mu_P(f)>2.$$
\end{rmk}

We finally deduce from Theorem~\ref{thmB_intro} the fact that $\lct(f,J_f^2)$ detects rational singularities.

\begin{proof}[Proof of Theorem~\ref{thmA_intro}]
It follows from
\cite[Theorem~0.4]{Saito-B} that the hypersurface $H$ defined by $f$ does not have rational singularities if and only if
$\widetilde{\alpha}_f\leq 1$, in which case we have $\widetilde{\alpha}_f=\lct(f)$. We deduce from
Theorem~\ref{thmB_intro} that in this case $\lct(f)\geq \lct(f,J_f^2)$, while the reverse inequality simply follows from the inclusion
$(f)\subseteq (f, J_f^2)$. This proves i).

The assertion in ii) is straightforward: in fact, we have the more general statement below (note that since the irreducible components of $H$ do not intersect,
we may assume that $H$ is irreducible, and we have
$(f)\subsetneq (f, J_f^2)$ since $H$ is reduced).
\end{proof}

\begin{prop}\label{rat_lct_maximal}
If $f\in\cO_X(X)$ defines an irreducible hypersurface $H$ with rational singularities, then for every coherent ideal $\fra$ of $\cO_X$ with $(f)\subsetneq\fra$, we have
$\lct(\fra)>\lct(f)=1$.
\end{prop}

\begin{proof}
Let $\pi\colon Y\to X$ be a log resolution of $(X,H)$ that is at the same time a log resolution
of the ideal $\fra$. Note first that if $E$ is a prime divisor on $Y$ such that ${\rm ord}_E(\fra)>0$, then
$E$ is a $\pi$-exceptional divisor. We know that
$H$ has rational singularities if and only if it has canonical singularities by a result of Elkik (see \cite[Theorem~11.1]{Kollar}); moreover,
this is the case if and only if the pair
$(X,H)$ has canonical singularities by a result of Stevens (see \cite[Theorem~7.9]{Kollar}). Since $(X,H)$ has canonical singularities and
$E$ is exceptional, we have
$$A_X({\rm ord}_E)\geq {\rm ord}_E(f)+1\geq {\rm ord}_E(\fra)+1.$$
This holds for all prime divisors $E$ on $Y$ for which ${\rm ord}_E(\fra)>0$, hence we conclude that
$\lct(\fra)>1$.
\end{proof}

\begin{rmk}\label{analytic}
While we worked in the algebraic setting, all results in this section have analogues when $X$ is a complex manifold and $f$ is a holomorphic function on $X$.
Indeed, note that \cite[Corollary~1.5]{DM} that was used in the proof of Proposition~\ref{gen_combination} also holds in the analytic setting (the proof in
${\mathit loc.\,cit.}$ applies verbatim in this setting). All the other arguments in this section extend to the analytic context in a straightforward way.
\end{rmk}

\begin{rmk}\label{rmk_same_valuations}
The key ingredient for the results in this section was the following consequence of Proposition~\ref{prop2_formal_equiv}:
if $X$ is smooth, with algebraic coordinates $x_1,\ldots,x_n$, $f\in\cO_X(X)$ is nonzero, having a singular point at $P$, and
 $g_{\lambda}=\lambda_0 f+\sum_{i,j=1}^n\lambda_{i,j}\tfrac{\partial f}{\partial x_i}\tfrac{\partial f}{\partial x_j}$,
 then there is an open subset $U\subseteq \C^{n^2+1}$ containing $(1,0,\ldots,0)$ such that any two $g_{\lambda}$ with 
 $\lambda\in U$ are formally equivalent at $P$. In particular, this implies that $\lct_P(f)=\lct_P(g_{\lambda})$ for all $\lambda\in U$.
In fact, it also implies that if this log canonical threshold is $<1$, then 
a divisor $E$ over $X$ computes $\lct_P(f)$ if and only if it computes $\lct_P(g_{\lambda})$: this follows from the general theorem below
(see also Remark~\ref{rmk_same_valuations}),
but it will not be used later.
\end{rmk}

\begin{thm}\label{prop_same_valuations}
Let $X$ be a smooth $n$-dimensional complex  algebraic variety and $P$ a point in $X$. Consider a nonzero ideal $\fra\subseteq\cO_X$ generated
by $f_1,\ldots,f_r\in\cO_X(X)$ and for every $\lambda=(\lambda_1,\ldots,\lambda_r)\in {\mathbf C}^r$
put
$f_{\lambda}=\lambda_1f_1+\ldots+\lambda_rf_r$. If 
 $U\subseteq {\mathbf C}^r$ is a nonempty open subset such that all pairs
$\big(X,{\rm div}(f_{\lambda})\big)$, with $\lambda\in U$, are isomorphic, then
for every $c$, with $0<c\leq\min\{\lct(\fra),1\}$, and every divisorial valuation ${\rm ord}_E$ on $X$ such that
\begin{equation}\label{eq0_same_val}
A_X({\rm ord}_E)-c\cdot {\rm ord}_E(f_{\lambda_0})<1-c,
\end{equation}
for some $\lambda_0\in U$, we have
${\rm ord}_E(f_{\lambda})={\rm ord}_E(f_{\lambda_0})$ for every $\lambda\in U$. 
In particular, if $\lct(\fra)<1$, then for every $\lambda,\lambda'\in U$, 
 a divisorial valuation ${\rm ord}_E$ computes ${\rm lct}(f_{\lambda})$ if and only if it computes
 $\lct(f_{\lambda'})$. 
\end{thm}

\begin{proof}
If the inequality (\ref{eq0_same_val}) holds for $c$, then it also holds for $c'>0$ with $0<c-c'\ll 1$. 
After replacing $c$ by such $c'$, 
we may assume that 
$c<\min\{\lct(\fra),1\}=\lct(f_{\lambda})$ for $\lambda\in U$.

For every $\lambda\in U$ and every $a\leq 1-c$, let 
$D(X,f_{\lambda}^c,a)$ be the set of divisorial valuations ${\rm ord}_G$ on $X$ that satisfy 
$$
A_X({\rm ord}_G)-c\cdot {\rm ord}_G(f_{\lambda})<a.
$$
Note that since $c<\lct(f_{\lambda})$, the set $D(X,f_{\lambda}^c, a)$ is finite by 
\cite[Proposition~2.36(2)]{KollarMori}. Since by assumption any two pairs $(X,f_{\lambda})$ and
$(X,f_{\lambda'})$ with $\lambda,\lambda'\in U$ are isomorphic, all sets $D(X, f_{\lambda}^c, a)$ with $\lambda\in U$ have the same number of elements. 

Let $\pi\colon Y\to X$ be a log resolution of $(X,\fra)$ and let us write $\fra\cdot\cO_Y=\cO_Y(-F)$. 
Since $c<{\rm lct}(\fra)$ is klt,
we may and will assume that if we write
$$c\cdot F-K_{Y/X}=A-B,$$
where $A$ and $B$ are effective divisors without common components, then $A$ is smooth (see \cite[Proposition~2.36(1)]{KollarMori}). 
There is a nonempty open subset $V\subseteq U$ such that for every $\lambda\in V$, we have
$$\pi^*\big({\rm div}(f_{\lambda})\big)=F+Z_{\lambda},$$
with $Z_{\lambda}$ smooth, without common components with $F$ and $K_{Y/X}$, and having SNC with these divisors. 

The key point is that if $\lambda\in V$ and a divisorial valuation ${\rm ord}_G$ lies in 
$D(X,f_{\lambda}^c,a)$, then $G$ is in fact a divisor on $Y$ (and it does not appear in $Z_{\lambda}$);
this follows, for example, from the fact that $A$ is smooth, $a\leq 1-c$,  and \cite[Corollary~2.31(3)]{KollarMori}). In particular, we have ${\rm ord}_G(\fra)={\rm ord}_G(f_{\lambda})$. 
Note also that for every $\lambda'\in U$, we have 
$$
{\rm ord}_G(f_{\lambda'})\geq {\rm ord}_G(\fra)
$$
and thus
$$A_X({\rm ord}_G)-c\cdot {\rm ord}_G(f_{\lambda'})\leq A_X({\rm ord}_G)-c\cdot {\rm ord}_G(\fra)=
A_X({\rm ord}_G)-c\cdot {\rm ord}_G(f_{\lambda})<a.$$
This implies that 
\begin{equation}\label{eq_same_val}
D(X,f_{\lambda}^c,a)\subseteq D(X,f_{\lambda'}^c,a)
\end{equation}
and since these are finite sets, with the same number of elements, we have equality in (\ref{eq_same_val}). 
Therefore for every $a\leq 1-c$, the set $D(X,f^c_{\lambda},a)$ with  $\lambda\in U$ does not depend on $\lambda$. 

Given a divisorial valuation ${\rm ord}_E\in D(X,f^c_{\lambda}, 1-c)$, with $\lambda\in U$, by considering all $a$ with
$$A_X({\rm ord}_E)-c\cdot {\rm ord}_E(f_{\lambda})<a\leq 1-c,$$
we conclude from the fact that we have equality in (\ref{eq_same_val}) that for every $\lambda'\in U$ we have
$$A_X({\rm ord}_E)-c\cdot {\rm ord}_E(f_{\lambda})\geq A_X({\rm ord}_E)-c\cdot {\rm ord}_E(f_{\lambda'}),$$
hence ${\rm ord}_E(f_{\lambda'})\geq {\rm ord}_E(f_{\lambda})$. By symmetry, this is in fact an equality, giving the first assertion in the theorem. 

If $c=\lct(\fra)<1$, then the set of divisorial valuations ${\rm ord}_E$ that compute $\lct(f_{\lambda})$ is precisely
$$\bigcap_{0<a\leq 1-c}D(X,f_{\lambda}^c,a).$$
Therefore the last assertion in the theorem follows from the first one.
\end{proof}

\begin{rmk}\label{rmk_same_valuations}
We have stated the proposition in the setting of this paper, but the reader will have no trouble seeing that the same proof gives the assertion 
more generally for a linear system of $\Q$-Cartier divisors on any variety with klt singularities. We also note that if instead of assuming 
that the pairs $\big(X,{\rm div}(f_{\lambda})\big)$, with $\lambda\in U$, are isomorphic, we only require them to be formally equivalent at some $P\in X$,
the assertions in the theorem hold for the divisorial valuations with center at $P$. In particular, if $\lct_P(\fra)<1$, then for every $\lambda,\lambda'\in U$, 
 a divisorial valuation ${\rm ord}_E$ computes ${\rm lct}_P(f_{\lambda})$ if and only if it computes
 $\lct_P(f_{\lambda'})$. This follows by replacing $X$ with ${\rm Spec}(\widehat{\cO_{X,P}})$ (for the relation between the invariants of valuations over $X$
 and those of the corresponding valuations over ${\rm Spec}(\widehat{\cO_{X,P}})$, see for example \cite[Proposition~A.14 and Remark~A.15]{dFEM2}). 

\end{rmk}

\section{An approach to $\lct(f,J_f^2)$ via arcs}\label{second_proof}

In this section we use the approach to valuations via arcs to prove Theorem~\ref{thm3_intro}, which we apply to
deduce Corollaries~\ref{corC_intro} and \ref{corD_intro} and give another proof of Theorem~\ref{thmA_intro}. We now assume that
$X$ is a smooth (irreducible)  variety, of dimension $n$, over an algebraically closed field $k$ of arbitrary characteristic.

We first review briefly the definition of jet and arc schemes. For details, see for example \cite{EM}.
For every $m\geq 0$, the $m^{\rm th}$ \emph{jet scheme} $X_m$ of $X$ is a scheme over $X$ with the property that for every
$k$-algebra $A$, we have a functorial bijection
$${\rm Hom}({\rm Spec} A,X_m)\simeq {\rm Hom}\big({\rm Spec} A[t]/(t^{m+1}),X\big).$$
In particular, the points of $X_m$ are in canonical bijection with the \emph{$m$-jets} on $X$, that is,
maps ${\rm Spec}\,k[t]/(t^{m+1})\to X$.
Given such an $m$-jet $\gamma\colon {\rm Spec}\,k[t]/(t^{m+1})\to X$,
we denote by $\gamma(0)$ the image of the closed point and by $\gamma^*$ the induced ring homomorphism
$\cO_{X,\gamma(0)}\to k[t]/(t^{m+1})$.
Truncation induces morphisms $X_m\to X_p$ whenever $p<m$ and these satisfy the obvious compatibilities.
Note that we have a canonical isomorphism  $X_0\simeq X$ and we denote by $\pi_m$ the truncation morphism
$X_m\to X$.
With this notation, we have $\pi_m(\gamma)=\gamma(0)$.

All truncation morphisms are affine, hence we may consider the projective limit $X_{\infty}$ of the system $(X_m)_{m\geq 0}$.
This is the \emph{space of arcs} (or \emph{arc scheme}) of $X$. Its $k$-valued points are in canonical bijection with maps $\Spec\,k\llbracket t\rrbracket\to X$.
In what follows, we identify $X_{\infty}$ with the corresponding set of $k$-valued points.
For an arc $\gamma$, we use the notation $\gamma(0)$ and $\gamma^*$ as above. Note that the space of arcs $X_{\infty}$ comes
endowed with truncation maps $\psi_m\colon X_{\infty}\to X_m$ compatible with the truncation morphisms between jet schemes.

Since $X$ is smooth and $n$-dimensional, every morphism $X_m\to X$ is locally trivial in the Zariski topology, with fiber $\A^{mn}$.
In fact, if $x_1,\ldots,x_n$ are algebraic coordinates on an open subset $U$ of $X$,
then we have an isomorphism
$$\pi_m^{-1}(U)\simeq U\times(\A^m)^n,$$
which maps $\gamma$ to $\big(\gamma(0),\gamma^*(x_1),\ldots,\gamma^*(x_n)\big)$
(note that each $\gamma^*(x_i)$ lies in $t k[t]/t^{m+1}k[t]\simeq k^m$).
In particular, every $X_m$ is a smooth, irreducible variety, of dimension $(m+1)n$.
Moreover, using the above isomorphisms we see that each truncation morphism $X_m\to X_p$,
with $p<m$, is locally trivial, with fiber $\A^{(m-p)n}$.

We next turn to the connection between divisorial valuations and certain subsets in the space of arcs.
A \emph{cylinder} in $X_{\infty}$ is a subset of the form $C=\psi_m^{-1}(S)$, where $S$ is a constructible subset of $X_m$.
In this case $C$ is \emph{irreducible}, \emph{closed}, \emph{open}, or \emph{locally closed}
(with respect to the Zariski topology on $X_{\infty}$)
if and only if $S$ has this property. In particular, we have irreducible decomposition
for locally closed cylinders.
The \emph{codimension}
of a cylinder $C=\psi_m^{-1}(S)$ is defined as
$${\rm codim}(C)={\rm codim}_{X_m}(S).$$
It is clear that this is independent of the way we write $C$ as the inverse image of a constructible set.

An important example of cylinders is provided by \emph{contact loci}. Given a coherent ideal ${\mathfrak a}$ in $\cO_X$,
we put ${\rm ord}_{\gamma}({\mathfrak a})\in\Z_{\geq 0}\cup\{\infty\}$ to be the $m$ such that $\gamma^*({\mathfrak a})k\llbracket t\rrbracket
=(t^m)$, with the convention that
${\rm ord}_{\gamma}({\mathfrak a})=\infty$ if $\gamma^*({\mathfrak a})=0$. The set ${\rm Cont}^{\geq m}({\mathfrak a})$ of $X_{\infty}$
consisting of all arcs $\gamma$ with ${\rm ord}_{\gamma}({\mathfrak a})\geq m$ is a closed cylinder in $X_{\infty}$. We similarly define the
locally closed cylinder ${\rm Cont}^m({\mathfrak a})$.

It turns out that one can use cylinders in $X_{\infty}$ in order to describe divisorial valuations on $X$. We simply state the results
and refer for details and proofs to \cite{ELM} for the characteristic $0$ case and to \cite{Zhu} for the general case (which only assumes that
the ground field $k$ is perfect).
 We assume, for simplicity, that $X$ is affine, though everything extends to the general case
in a straightforward way. For every closed, irreducible cylinder $C\subseteq X_{\infty}$, that does not dominate
$X$ and every $h\in\cO_X(X)$ nonzero, we put
$${\rm ord}_C(h):=\min\{{\rm ord}_{\gamma}(h)\mid\gamma\in C\}\in\Z_{\geq 0}.$$
This extends to a valuation of the function field of $X$; in fact, this is a divisorial valuation, whose center is the
closure of $\psi_0(C)$.

Conversely, if $v=q\cdot {\rm ord}_E$ for some smooth prime divisor $E$ on the smooth variety $Y$, with a birational
morphism $g\colon Y\to X$,
then we have an induced morphism $g_{\infty}\colon Y_{\infty}\to X_{\infty}$ and if $C(v)$ is the closure of $g_{\infty}\big({\rm Cont}^{\geq q}(\cO_Y(-E))\big)$,
then $C(v)$ is an irreducible closed cylinder in $X_{\infty}$ and ${\rm ord}_{C(v)}=v$. Moreover, a key fact is that
$${\rm codim}\big(C(v)\big)=A_X(v),$$
see \cite[Theorem~A(2)]{Zhu}.
In addition, for every closed, irreducible cylinder $C\subseteq X_{\infty}$ that does not dominate $X$, if $v={\rm ord}_C$, then
$C\subseteq C(v)$. In particular, we have ${\rm codim}(C)\geq A_X(v)$.

After this overview, we can prove our general result about valuations.

\begin{proof}[Proof of Theorem~\ref{thm3_intro}]
After possibly replacing $X$ by an affine open subset that intersects nontrivially $C_X(v)$, we may and will assume that $X$ is affine.
Let $C=C(v)$, so that ${\rm ord}_C=v$ and ${\rm codim}(C)=A_X(v)$.
We put $Z=c_X(v)$, so that $Z$ is the closure of $\psi_0(C)$.
If $m=v(f)$ and $e=v(J_f)$,
we have by hypothesis $0<e<\frac{1}{2}m$. Since $m={\ord}_C(f)$ and $e={\rm ord}_C(J_f)$, we have $C\subseteq {\rm Cont}^{\geq m}(f)\cap
{\rm Cont}^{\geq e}(J_f)$.
Let $C_0:=C\cap {\rm Cont}^e(J_f)$, which is a nonempty subcylinder of $C$, open in $C$.
Since $C$ is irreducible, we have $C=\overline{C_0}$.
We also consider the locally closed cylinder
$$C':={\rm Cont}^{\geq (m-1)}(f)\cap{\rm Cont}^e(J_f)\cap \psi_0^{-1}(Z).$$
It is clear that $C_0$ is a closed subset of $C'$. We make the following

\noindent {\bf Claim}. $C_0$ is \emph{not} an irreducible component of $C'$.

Assuming the claim, let $W$ be an irreducible component of $C'$ that contains $C_0$ and
$\overline{W}$ its closure in $X_{\infty}$. Note that $\overline{W}$ is an irreducible, closed cylinder
in $X_{\infty}$ such that $\psi_0(\overline{W})\subseteq Z$. Since we also have
$$Z\subseteq\overline{\psi_0(C_0)}\subseteq\overline{\psi_0(\overline{W})},$$
we conclude that $\overline{\psi_0(\overline{W})}=Z$. Therefore $w:={\rm ord}_{\overline{W}}$
is a divisorial valuation on $X$, with center $Z$.

Since $W\subseteq {\rm Cont}^{\geq m-1}(f)$, it follows that $w(f)\geq m-1$. Furthermore, since
$C_0\subseteq W$, we have $C=\overline{C_0}\subseteq\overline{W}$, hence we clearly have
$$w(g)={\rm ord}_{\overline{W}}(g)\geq {\rm ord}_C(g)=v(g)\quad\text{for all}\quad g\in\cO_X(X).$$
Finally, since $C_0$ is a proper closed subset of $W$, we have
$$A_X(w)\leq {\rm codim}(\overline{W})\leq {\rm codim}(\overline{C_0})-1=A_X(v)-1,$$
hence $w$ satisfies all the required conditions. Therefore it is enough to prove the claim.

Note that our assumptions imply that $m-e-2\geq e-1\geq 0$. In order to simplify the notation, we put $\psi=\psi_{m-e-2}\colon
X_{\infty}\to X_{m-e-2}$.
The key point is to describe, for every $\overline{\gamma}\in\psi(C')$, the cylinders
$\psi^{-1}(\overline{\gamma})\cap C_0\subseteq \psi^{-1}(\overline{\gamma})\cap C'$.
This is a local computation, based on Taylor's formula, that goes back to the proof of \cite[Lemma~3.4]{DenefLoeser}.
We choose an algebraic system of coordinates $x_1,\ldots,x_n$ in  a neighborhood of $P=\pi_{m-e-2}(\overline{\gamma})$,
centered at $P$. This gives an isomorphism $\widehat{\cO_{X,P}}\simeq k\llbracket y_1,\ldots,y_n\rrbracket$ that maps each $x_i$ to $y_i$, and
let $\varphi\in k\llbracket y_1,\ldots,y_n\rrbracket$ be the formal power series corresponding to $f$.
As we have seen, the system of coordinates also allows us to identify an arc on $X$ lying over $P$ with an element of $\big(t k\llbracket t\rrbracket\big)^n$.

Let $\gamma\in C'$ be an arc with $\psi(\gamma)=\overline{\gamma}$, which corresponds to $u=(u_1,\ldots,u_n)\in \big(t k\llbracket t\rrbracket\big)^n$,
so that $\gamma^*(f)=\varphi(u_1,\ldots,u_n)$. Any other arc $\delta\in \psi^{-1}(\overline{\gamma})$ corresponds to
$u+v$, for some $v=(v_1,\ldots,v_n)\in \big(t^{m-e-1} k\llbracket t\rrbracket\big)^n$. Applying Lemma~\ref{Taylor_formula} for 
$\varphi$, we get
\begin{equation}\label{eq_thm3_intro}
\delta^*(f)=\varphi(u+v)=\gamma^*(f)+\sum_{i=1}^n\gamma^*\left(\frac{\partial f}{\partial x_i}\right)v_i+\sum_{1\leq i\leq j\leq n}\gamma^*\big(D_{i,j}(f)\big)v_iv_j+\text{higher order terms},
\end{equation}
where $D_{i,j}=\partial_{x_i}\partial_{x_j}$ if $i<j$ and $D_{i,i}=\tfrac{\partial _{x_i}^2}{2}$. 
By assumption, we have ${\rm ord}_t\gamma^*(f)\geq m-1$ and
$$\min_{1\leq i\leq n}{\rm ord}_t\gamma^*\left(\frac{\partial f}{\partial x_i}\right)=e$$
(note that $m-1>e$). Since ${\rm ord}_tv_i\geq m-e-1$, an immediate computation using (\ref{eq_thm3_intro})
shows that for every such $\delta$, we have ${\rm ord}_t\delta^*(f)\geq m-1$. Since $m-e-1\geq e$, we also see that
${\rm ord}_t\delta^*(J_f)\geq e$, hence $\psi^{-1}(\overline{\gamma})\cap C'$ is the open subset of $\psi^{-1}(\overline{\gamma})$ defined by
having contact order with $J_f$ precisely $e$; in particular, this is an irreducible cylinder.

For every series $g\in k\llbracket t\rrbracket$, let us write $g_0$ for the constant term of $g$.
We also write $v_{i,0}$ for the constant term of $v_i/t^{m-e-1}$.
It follows from (\ref{eq_thm3_intro})
that the coefficient of $t^{m-1}$ in $\delta^*(f)$ is equal to
$$\big(\gamma^*(f)/t^{m-1}\big)_0+\sum_{i=1}^n\big(\gamma^*(\partial f/\partial x_i)/t^e\big)_0v_{i,0}+
\sum_{1\leq i\leq j\leq n}\big(\gamma^*(D_{i,j}(f))\big)_0v_{i,0}v_{j,0},$$
and the last term only appears if $m=2e+1$. In any case, since some $\big(\gamma^*(\partial f/\partial x_i)/t^e\big)_0$ is nonzero,
the vanishing of this coefficient defines a hypersurface in the affine space $\A^n$ parametrizing $(v_{1,0},\ldots,v_{n,0})$.
Since $C_0\subseteq {\rm Cont}^{\geq m}(f)$, it follows that $\psi^{-1}(\overline{\gamma})\cap C_0$ is different from $\psi^{-1}(\overline{\gamma})\cap C'$.
The fact that $C_0$ is not an irreducible component of $C'$ follows by taking the image in some $X_q$, with $q\gg 0$
(so that $C_0$ is the inverse image of a locally closed subset of $X_q)$, and applying the following easy lemma.
\end{proof}

\begin{lem}
Let $g\colon U\to V$ be a morphism of algebraic varieties and $W$  a closed subset of $U$
such that for every $y\in g(U)$, the following two conditions hold:
\begin{enumerate}
\item[i)] The fiber $g^{-1}(y)$ is irreducible, and
\item[ii)] The intersection $g^{-1}(y)\cap W$ is different from $g^{-1}(y)$.
\end{enumerate}
Then $W$ is not an irreducible component of $U$.
\end{lem}

\begin{proof}
Arguing by contradiction, suppose that the irreducible components of $U$ are
$U_1=W,U_2,\ldots,U_r$. If $x\in W\smallsetminus\bigcup_{i\geq 2}U_i$ and $y=g(x)$, then
$g^{-1}(y)\not\subseteq U_i$ for any $i\geq 2$. On the other hand,
$g^{-1}(y)$ is irreducible and contained in $\bigcup_{i\geq 1}\big(U_i\cap g^{-1}(y)\big)$,
hence $g^{-1}(y)\subseteq U_1$. This contradicts condition ii).
\end{proof}

From now on we assume that $k=\C$ (in fact, everything works if ${\rm char}(k)=0$).
We can now deduce the assertion about multiplier ideals:

\begin{proof}[Proof of Theorem~\ref{corC_intro}]
Since $(f)\subseteq (f,J_f^2)$, we clearly have the inclusion
$$\cJ(X, f^{\lambda})\subseteq\cJ\big(X, (f,J_f^2)^{\lambda}\big)$$
for all $\lambda>0$. We now suppose that $\lambda<1$ and prove the reverse inclusion.

We may and will assume that $X$ is affine. Arguing by contradiction, suppose that we have $g\in \cJ\big(X, (f,J_f^2)^{\lambda}\big)$
such that $g\not\in\cJ(X, f^{\lambda})$. By definition of multiplier ideals, the latter condition
implies  that there is a valuation $v={\rm ord}_E$,
where $E$ is a prime divisor on a log resolution of $(X,f)$ such that
\begin{equation}\label{eq_corC}
v(g)+A_X(v)\leq \lambda\cdot v(f).
\end{equation}
We choose such $E$ with the property that $A_X(v)$ is minimal.

If $v(J_f)\geq \frac{1}{2}v(f)$, then $v(f,J_f^2)=\min\big\{v(f), 2\cdot v(J_f)\big\}=v(f)$, and
(\ref{eq_corC}) contradicts the fact that $g\in  \cJ\big(X, (f,J_f^2)^{\lambda}\big)$.
Hence we may and will assume that $v(J_f)<\frac{1}{2}v(f)$. If $v(J_f)=0$, then there is an open subset $U$ of $X$
that intersects the center of $v$ on $X$ and such that $f\vert_U$ defines a smooth hypersurface. In this case, since $\lambda<1$,
we have $\cJ(U, f\vert_U^{\lambda})=\cO_U$. On the other hand, since $U$ intersects the center of $v$, it follows from (\ref{eq_corC})
that $g\vert_U\not\in \cJ(U, f\vert_U^{\lambda})$, a contradiction.

We thus may and will assume that
$$0<v(J_f)<\tfrac{1}{2}v(f).$$
In this case, it follows from Theorem~\ref{thm3_intro} that there is a divisorial valuation $w$ on $X$ that satisfies
properties i), ii), and iii) in the theorem. We thus have
$$A_X(w)\leq A_X(v)-1\leq \lambda\cdot v(f)-1-v(g)\leq \lambda\cdot \big(w(f)+1\big)-1-w(g),$$
where the first inequality follows from condition iii) and the third inequality follows from conditions i) and ii).
Since
$$\lambda\cdot \big(w(f)+1\big)-1-w(g)=\lambda\cdot w(f)-w(g)+\lambda-1\leq \lambda\cdot w(f)-w(g),$$
if we write $w=q\cdot {\rm ord}_F$, for some prime divisor $F$ over $X$, after dividing the above inequalities by $q$,
we obtain
$$A_X({\rm ord}_F)\leq \lambda\cdot {\rm ord}_F(f)-{\rm ord}_F(g)$$
and
$$A_X({\rm ord}_F)\leq A_X(w)\leq A_X(v)-1,$$
contradicting the minimality in the choice of $E$. The contradiction we obtained shows that we have in fact
$$\cJ\big(X, (f,J_f^2)^{\lambda}\big)\subseteq \cJ(X, f^{\lambda}),$$
completing the proof of the first assertion in the theorem.

The proof of the second assertion is similar. We may and will assume that $X$ is affine. Recall that the
adjoint ideal ${\rm adj}(f)$ consists of all $g\in\cO_X(X)$ with the property that for every exceptional divisor $E$ over $X$,
we have
$${\rm ord}_E(g)>{\rm ord}_E(f)-A_X({\rm ord}_E)$$
(see \cite[Chapter~9.3.48]{Lazarsfeld}).
The inclusion
$$
{\rm adj}(f)\subseteq\cJ\big(X, (f,J_f^2)\big)
$$
is easy: if $g\in {\rm adj}(f)$, then for every divisorial valuation ${\rm ord}_E$ on $X$, we have
\begin{equation}\label{eq0_pf_cor}
{\rm ord}_E(g)>{\rm ord}_E(f,J_f^2)-A_X({\rm ord}_E).
\end{equation}
If $E$ is exceptional, this follows from the definition of the adjoint ideal and the fact that ${\rm ord}_E(f)\geq {\rm ord}_E(f,J_f^2)$.
On the other hand, if $E$ is a divisor on $X$, then $A_X({\rm ord}_E)=1$ and ${\rm ord}_E(f,J_f^2)=0$ (the latter equality follows from the fact that
if $f$ vanishes on $E$, since $f$ defines a reduced hypersurface, we have ${\rm ord}_E(J_f)=0$). Since (\ref{eq0_pf_cor})
holds for every $E$, we conclude that $g\in\J\big(X, (f,J_f^2)\big)$.

We now turn to the interesting inclusion
$$\cJ\big(X, (f,J_f^2)\big)\subseteq {\rm adj}(f).$$
Suppose that $g\in \cJ\big(X, (f,J_f^2)\big)$, but $g\not\in {\rm adj}(f)$. The latter condition implies that there is
an exceptional divisor $E$ over $X$ such that
\begin{equation}\label{eq_pf_cor}
{\rm ord}_E(g)\leq {\rm ord}_E(f)-A_X({\rm ord}_E).
\end{equation}
We choose such $E$ with $A_X({\rm ord}_E)$ minimal
and argue as in the proof of the first part.

If ${\rm ord}_E(J_f)\geq\frac{1}{2}{\rm ord}_E(f)$, then ${\rm ord}_E(f)={\rm ord}_E(f,J_f^2)$, and
(\ref{eq_pf_cor}) contradicts the fact that $g\in\cJ\big(X, (f,J_f^2)\big)$.
Hence we may and will assume that ${\rm ord}_E(J_f)<\frac{1}{2}{\rm ord}_E(f)$.
If ${\rm ord}_E(J_f)=0$, then there is an open subset $U$ of $X$ that intersects the center of ${\rm ord}_E$ on $X$
and such that $f\vert_U$ defines a smooth hypersurface. In particular,
we have ${\rm adj}(f)\vert_U=\cO_U$ by \cite[Proposition~9.3.48]{Lazarsfeld}, contradicting (\ref{eq_pf_cor}).

We thus may and will assume that
$$0<{\rm ord}_E(J_f)<{\rm ord}_E(f).$$
We then apply Theorem~\ref{thm3_intro} for $v={\rm ord}_E$
 to find a divisorial valuation $w=q\cdot {\rm ord}_F$ on $X$ that satisfies
properties i)-iv) in the theorem. We obtain
$$w(g)\leq {\rm ord}_E(g)\leq {\rm ord}_E(f)-A_X({\rm ord}_E)\leq w(f)-A_X(w).$$
Dividing by $q$, we obtain
$${\rm ord}_F(g)\leq {\rm ord}_F(f)-A_X({\rm ord}_F).$$
Since
$A_X({\rm ord}_F)\leq A_X(w)\leq A_X(v)-1$ and $F$ is an exceptional divisor
(we use here the fact that $c_X({\rm ord}_F)=c_X({\rm ord}_E)$), this contradicts the minimality in our choice of $E$.
This completes the proof of the theorem.
\end{proof}

Theorem~\ref{corC_intro} easily implies the assertions in Theorem~\ref{thmA_intro}.

\begin{proof}[Second proof of Theorem~\ref{thmA_intro}]
Recall that for a proper nonzero ideal ${\mathfrak a}$, we have
$$\lct({\mathfrak a})=\min\{\lambda>0\mid\cJ(X, {\mathfrak a}^{\lambda})\neq\cO_X\}.$$
Since $\lct(f)\leq 1$, the first assertion in Corollary~\ref{corC_intro} implies that in general, we have
$$\lct(f)=\min\{\lct(f,J_f^2),1\}.$$
Note that if $\lct(f)=1$, then automatically $H$ is reduced. We thus deduce from the second
assertion in Corollary~\ref{corC_intro} that ${\rm lct}(f,J_f^2)>1$ if and only if $H$ is reduced
and ${\rm adj}(f)=\cO_X$. By \cite[Proposition~9.3.48]{Lazarsfeld}, this holds if and only if $H$ has rational singularities.
\end{proof}

We end this section by illustrating how to use Theorem~\ref{thm3_intro} to handle arbitrary ideals $\fra$, with $\lct(\fra)<1$.

\begin{proof}[Proof of Corollary~\ref{corD_intro}]
The inequality $\lct(\fra)\leq\lct\big(\fra+D(\fra)^2\big)$ is clear, hence we only need to prove the reverse inequality. We may and will
assume that $X$ is affine. Let $f$ be a general linear combination, with coefficients in ${\mathbf C}$, of the generators of $\fra$.
In this case, it follows from \cite[Proposition~9.2.28]{Lazarsfeld}) that $\lct(f)=\lct(\fra)<1$. Moreover, it follows from the proof in ${\mathit loc.\,cit.}$
that a divisorial valuation $v$ computes $\lct(f)$ if and only if it computes $\lct(\fra)$.

We next note that for every such $f$, if $v$ computes $\lct(f)$, then $v(J_f)\geq \frac{v(f)}{2}$. Indeed, suppose that we have
$v(J_f)<\frac{v(f)}{2}$. Note that we also have $v(J_f)>0$ (otherwise $c_X(v)$ intersects the smooth locus of the hypersurface defined by $f$,
contradicting the fact that $\lct(f)<1$). We can thus apply Theorem~\ref{thm3_intro} to find a divisorial valuation $w$ on $X$ such that
$A_X(w)\leq A_X(v)-1$ and $w(f)\geq v(f)-1$. In this case we have
$$\frac{A_X(w)}{w(f)}\leq \frac{A_X(v)-1}{v(f)-1}<\frac{A_X(v)}{v(f)}=\lct(f),$$
a contradiction (note that the strict inequality follows from the fact that $\lct(f)<1$ implies $A_X(v)<v(f)$).

We thus conclude that if $v$ is a valuation that computes $\lct(\fra)$, then for every general linear combination $f$ as above, we have
$v(J_f)\geq \frac{v(f)}{2}=\frac{v(\fra)}{2}$. This implies $v\big(D(\fra)\big)\geq \frac{v(\fra)}{2}$.
We thus conclude that
$$\lct\big(\fra+D(\fra)^2\big)\leq \frac{A_X(v)}{v\big(\fra+D(\fra)^2\big)}=\frac{A_X(v)}{v(\fra)}=\lct(\fra),$$
completing the proof of the corollary.
\end{proof}

\section{Two examples}\label{examples}

In this section we discuss two examples of complex hypersurfaces. We begin with the case of the generic determinantal hypersurface,
for which we show that the inequality in Theorem~\ref{thmB_intro} is an equality.

\begin{eg}\label{eg_det}
Let $n\geq 2$ and $X=\A^{n^2}$ be the affine space of $n\times n$ matrices, with coordinates $x_{i,j}$, for $1\leq i,j\leq n$.
We consider $f={\rm det}(A)$, where $A$ is the matrix $(x_{i,j})_{1\leq i,j\leq n}$. In this case,
 it is well-known that the Bernstein-Sato polynomial of $f$ is given by
$$b_f(s)=\prod_{i=1}^n(s+i)$$
(see, for example, \cite[Appendix]{Kimura}). We thus have $\widetilde{\alpha}(f)=2$.

In order to compute $\lct(f,J_f^2)$, we use the arc-theoretic description reviewed in the previous section, together with
the approach in the determinantal case due to Docampo \cite{Docampo}. Note first that $J_f$ is the ideal of $\cO_X(X)$
generated by the $(n-1)$-minors of the matrix $A$. We use the action of $G={\rm GL}_n(\C)\times {\rm GL}_n(\C)$ on
$X$ given by $(g,h)\cdot A=gAh^{-1}$. We have an induced action of $G_{\infty}={\rm GL}_n\big(\C\llbracket t\rrbracket\big)\times
{\rm GL}_n\big(\C\llbracket t\rrbracket\big)$
on $X_{\infty}$ and we consider the orbits with respect to this action. Since the ideal $(f,J_f^2)$ is preserved by the $G$-action, it follows that
every contact locus of this ideal is a union of $G_{\infty}$-orbits.

Recall that a $G_{\infty}$-orbit in $X_{\infty}$ of finite codimension corresponds to a sequence of integers $\lambda_1\geq\ldots,\geq\lambda_n\geq 0$,
such that the orbit consists of those $n\times n$ matrices with entries in $\C\llbracket t\rrbracket$ that are equivalent via Gaussian elimination to the diagonal matrix 
${\rm diag}(t^{\lambda_1},\ldots,t^{\lambda})$,
having on the diagonal $t^{\lambda_1},\ldots,t^{\lambda_n}$. The codimension of this orbit is $\sum_{i=1}^n\lambda_i(2i-1)$ by
\cite[Theorem~C]{Docampo}. Moreover, if $\gamma$ is an arc in this orbit, 
then
${\rm ord}_{\gamma}(f)=\sum_{i=1}^n\lambda_i$ and
$${\rm ord}_{\gamma}(J_f)=\min\{\lambda_1+\ldots+\widehat{\lambda_i}+\ldots+\lambda_n\mid 1\leq i\leq n\}=\lambda_2+\ldots+\lambda_n$$
(this follows by computing the determinant, respectively, the ideal generated by the $(n-1)$-minors of the matrix ${\rm diag}(t^{\lambda_1},\ldots,t^{\lambda_n})$).
Therefore the order of $\gamma$ along the ideal $(f,J_f^2)$ is
$$\min\left\{\sum_{i=1}^n\lambda_i, 2\cdot\sum_{i=2}^n\lambda_i\right\}.$$
The description of the log canonical threshold in terms of contact loci (see \cite{ELM}) thus implies
\begin{equation}\label{eq_eg1}
\lct(f,J_f^2)=\min_{\lambda}\frac{\sum_{i=1}^n\lambda_i(2i-1)}{\min\left\{\sum_{i=1}^n\lambda_i, 2\cdot\sum_{i=2}^n\lambda_i\right\}},
\end{equation}
where the minimum is over all $\lambda=(\lambda_1,\ldots,\lambda_n)$, with $\lambda_1\geq\ldots\geq\lambda_n\geq 0$
and $\lambda_2>0$. If we take $\lambda_1=\lambda_2>0$ and $\lambda_3=\ldots=\lambda_n=0$, then the expression on the right-hand side of
(\ref{eq_eg1}) is $2$. In order to see that for all $\lambda$ this expression is $\geq 2$,  note that since $\lambda_1\geq\lambda_2$, we have
$$\sum_{i=1}^n(2i-1)\lambda_i\geq 4\cdot\sum_{i=2}^n\lambda_i.$$
We thus conclude that $\lct(f,J_f^2)=2=\widetilde{\alpha}(f)$.
\end{eg}

We next turn to the case of homogeneous diagonal hypersurfaces, where we will see that the inequality
in Theorem~\ref{thmB_intro} can be strict.

\begin{eg}\label{eg_diag}
Let $X=\A^n$, with $n\geq 2$, and $f=x_1^d+\ldots+x_n^d$, for some $d\geq 2$. In this case, it is known that the Berstein-Sato polynomial of $f$ is given by
$$b_f(s)=(s+1)\cdot\prod_{1\leq b_1,\ldots,b_n\leq d-1}\left(s+\sum_{i=1}^n\frac{b_i}{d}\right)$$
(see \cite[Proposition~3.6]{Yano}). In particular, we have $\widetilde{\alpha}(f)=\frac{n}{d}$.

We will show that
\begin{equation}\label{eq_eg_diag}
\lct(f,J_f^2)=\min\left\{\frac{n+d-2}{2d-2},\frac{n}{d}\right\}.
\end{equation}
It is clear that we have $J_f=(x_1^{d-1},\ldots,x_n^{d-1})$. Let $\pi\colon Y\to X$ be the blow-up of $X$
at the origin. By symmetry, it is enough to consider the chart $U$
on $Y$ with coordinates $y_1,\ldots,y_n$ such that $x_1=y_1$ and $x_i=y_1y_i$ for $i\geq 2$.
In this chart we have $f\circ\pi\vert_U=y_1^dg$, with $g=1+y_2^d+\ldots+y_n^d$, and $J_f^2\cdot\cO_U=(y_1^{2d-2})$.
We deduce that a log resolution of the ideal $(f,J_f^2)\cdot\cO_U$ can be obtained by blowing-up $U$ along the ideal $(g,y_1^{d-2})$ and then resolving
torically the resulting variety. In particular, we see that in order to compute $\lct(f,J_f^2)$ it is enough to consider toric divisors over $Y$,
with respect to the system of coordinates given by $y_1$ and $g$. If $E$ is such a divisor, with ${\rm ord}_E(g)=a$ and ${\rm ord}_E(y_1)=b$,
with $a$ and $b$ nonnegative integers, not both $0$, we have
$${\rm ord}_E(f,J_f^2)=\min\{db+a, (2d-2)b\}\quad\text{and}\quad A_X({\rm ord}_E)=nb+a.$$
This implies that
$$\lct(f,J_f^2)=\min_{(a,b)}\frac{nb+a}{\min\{db+a,(2d-2)b\}},$$
where the minimum is over all $(a,b)\in\Z_{\geq 0}^2\smallsetminus\{(0,0)\}$.
It is now a straightforward exercise to deduce the formula (\ref{eq_eg_diag}).

If $d=2$, then we see that $\lct(f,J_f^2)=\frac{n}{2}=\widetilde{\alpha}(f)$ (of course, this can be seen directly, since in this case
$(f,J_f^2)=(x_1,\ldots,x_n)^2$). Suppose now that $d\geq 3$. In this case, a straightforward calculation shows that
$$\frac{n+d-2}{2d-2}<\frac{n}{d}\quad\text{if and only if}\quad d<n.$$
 We thus see that if $d\geq n$ (which is precisely the case when the hypersurface defined by $f$ does not have rational
singularities), then $\lct(f,J_f^2)=\frac{n}{d}=\lct(f)$, as expected according to Theorem~\ref{thmA_intro}. On the other hand, if $3\leq d<n$, then
we have a strict inequality $\lct(f,J_f^2)<\widetilde{\alpha}(f)$.
\end{eg}

\section{The motivic oscillation index and $\lct(f,J_f^2)$}\label{moi}

We begin by recalling the definition of the motivic oscillation index from \cite{CMN}.
Let $\Qalg$  be the algebraic closure of $\Q$ inside $\C$.
We consider a nonconstant polynomial $f$ in $\Qalg[x_1,\ldots,x_n]$ and a closed subscheme  $Z$ of $\A_{\Qalg}^n$.

 Let $K$ be a number field  such that $f$ and $Z$ are defined over $K$. Choose an integer $N>0$ such that $f$ and a set of generators for the ideal
 defining $Z$ lie in $\cO[1/N]$, where $\cO$ is the ring of integers of $K$. For any prime $p$ not dividing $N$, any completion $L$  of $K$ above $p$,
and any nontrivial additive character $\psi\colon L\to \C^\times$, consider the following integral
\begin{equation}\label{eq:EfZ0}
E_{f,L,\psi}^{Z} := \int_{\{x\in \cO_{L}^n\mid \overline x\in Z(k_L) \} }  \psi\big(f(x)\big)|dx| ,
\end{equation}
where $\cO_L$ stands for the valuation ring of $L$ with residue field $k_L$,
$\overline x$ stands for the image of $x$ under the natural projection $\cO_{L}^n \to k_L^n$, and $|dx|$ is the Haar measure on $L^n$ normalized so that
$\cO_L^n$ has measure $1$. Writing $q_L$ for the number of elements of $k_L$, let $\sigma_L$ be the supremum over all $\sigma\geq 0$ such that
$$
|E_{f,L,\psi}^{Z}| \leq c q_L^{-m\sigma}
$$
where $m$ is such that $\psi$ is trivial on $\cM_L^m$ and nontrivial on $\cM_L^{m-1}$ with  $\cM_L$ the maximal ideal of $\cO_L$, and where $c=c(\sigma,f,L,Z)$ is independent of $\psi$ and $m$.  Note that $\sigma_L$ can equal $+\infty$; this is the case precisely when
the morphism $\A_{\Qalg}^n\to\A_{\Qalg}^1$ defined by $f$ is smooth in an open neighborhood of $Z$.
We define the $K$-oscillation index of $f$ along $Z$ as
$$
\oi_Z(f) := \lim_{M\to\infty}\inf_{L} \sigma_L,
$$
where the infimum is taken over all non-Archimedean completions $L$ of $K$ above primes $p_L$ with $p_L>M$. 
Finally, the motivic oscillation index of $f$ along $Z$ is defined as
$$
\moi_Z(f) := \inf_{K} \oi_Z(f),
$$
where $K$ runs over all number fields satisfying the above conditions.
This definition corresponds to the definition given in \cite[Section~3.4]{CMN} by Igusa's work (see \cite{Denef}), which relates upper bounds for oscillating integrals with nontrivial poles of local zeta functions.
Note that the variant $\moi(f)$ we considered in the Introduction corresponds to the case when $Z$ is the hypersurface defined by $f$ (note the small change in notation for $\moi(f)$, compared to \cite{CMN}).

The following is the main result of this section.

\begin{thm}\label{thm_moi2}
If $f\in\Qalg[x_1,\ldots,x_n]$ is a nonconstant polynomial and $Z$ is a closed subscheme of the hypersurface defined by $f$, then
\begin{equation}\label{eq:moifJ2}
\moi_Z(f) \geq \lct_{Z}(f,J_f^2).
\end{equation}
In addition, if $\lct_{Z}(f,J_f^2)\leq 1$, then equality holds in (\ref{eq:moifJ2});
in this case, we also have $\oi_Z(f)  = \lct_{Z}(f,J_f^2)$
for every number field $K$ such that $f$ and $Z$ are defined over $K$.
\end{thm}

The proof of Theorem \ref{thm_moi2} is based on Proposition \ref{prop:E^ZfJf} below. The strong monodromy conjecture states that the real parts of the poles of non-archimedean local zeta functions are zeros of the Bernstein-Sato polynomial, see \cite[(2.3.1)]{Denef}.
From Proposition \ref{prop:E^ZfJf} we deduce furthermore Igusa's strong monodromy conjecture in the range strictly between $-\lct(f,J_f^2)$ and $0$, as follows.
\begin{cor}\label{cor:monodromy}
Let $f\in K[x_1,\ldots,x_n]$ be a nonconstant polynomial, with $K\ge \Q$ a number field. Let $L$ be a non-archimedean completion of $K$, $\chi:\cO_L^\times\to\C^\times$ a group homomorphism, and consider the local zeta function defined for $s>0$ by
$$
P_{f,L,\chi}(t) :=  \int_{x\in \cO_L^n}\chi(f(x)) |f(x)|^s |dx|,
$$
with $t=p^{-s}$. Note that  $P_{f,L,\chi}(t)$ is a rational function in $t$ by Igusa's work.  We suppose that the residue field characteristic of $L$ is large (depending on $f$). Then the strong monodromy conjecture holds for $f$ for values strictly between $-\lct(f,J_f^2)$ and $0$. In detail, let $t_0$ be the real part of a pole of $P_{f,L,\chi}(t)$. Suppose that  $t_0>- \lct(f,J_f^2)$. Then $t_0=-1$ and $t_0$ is a zero of the Bernstein-Sato polynomial of $f$.
\end{cor}

\begin{rmk}
Note that the condition in Corollary \ref{cor:monodromy} on the residue field characteristic of $L$ is removed in Section 4 of \cite{Nguyen:n/2d}.  
\end{rmk}

We recall that for every nonzero ideal ${\mathfrak a}$ in $\Qalg[x_1,\ldots,x_n]$, by definition
$$\lct_Z({\mathfrak a})=\max_U\lct({\mathfrak a}\vert_U),$$
where the maximum is over all open neighborhoods $U$ of $Z$. We also note that if ${\mathfrak a}_{\C}$
is the extension of ${\mathfrak a}$ to $\C[x_1,\ldots,x_n]$, then $\lct({\mathfrak a})=\lct({\mathfrak a}_{\C})$.
We will derive Theorem~\ref{thm_moi2} from the following proposition and the definition of the oscillation index.

\begin{prop}\label{prop:E^ZfJf}
Let $K$ be any number field such that $f$ and $Z$ are defined over $K$ and let
$L$ be a non-Archimedean completion of $K$ above a prime $p_L$, with residue field $k_L$ with $q_L$ elements.
If $p_L$ is large enough (in terms of the data $f,Z,K$), then for every $\varepsilon>0$, there exists a constant $c = c(f,Z,L,\varepsilon)$ such that for each nontrivial additive character $\psi\colon L\to\C^\times$ we have
\begin{equation}\label{eq:EZ}
|E^{Z}_{f,L,\psi}| <  c q_L^{- m\sigma }\quad\text{for}\quad
\sigma =  \lct_Z(f,J_f^2) - \varepsilon,
\end{equation}
where $m=m(\psi)$ is such that $\psi$ is trivial on $\cM_L^m$ and nontrivial on $\cM_L^{m-1}$.
\end{prop}

\begin{proof}[Proof of Proposition \ref{prop:E^ZfJf}]
By Igusa's results from \cite{Igusa}, \cite{Igusa3} and \cite{Denef}, recalled in \cite[Propositions~3.1 and 3.4]{CMN}, we have for each $\psi$ with $m=m(\psi)>1$
\begin{equation}\label{eq:EfZ1}
E_{f,L,\psi}^{Z} = \int_{\{x\in \cO_{L}^n\mid \overline x\in Z(k_L),\ \ord f(x)\geq m-1 \} }  \psi\big(f(x)\big)|dx| ,
\end{equation}
where $\overline x$ stands for the image of $x$ under the natural projection $\cO_{L}^n \to k_L^n$, and where $p_L$ is assumed to be large. We now show
that we also have
\begin{equation}\label{eq:EfZJ}
E_{f,L,\psi}^{Z}= \int_{ \{x\in \cO_{L}^n\mid \overline x\in Z(k_L),\ \ord(f(x),J_f^2(x))\geq m-1 \} }  \psi\big(f(x)\big)|dx|,
\end{equation}
by adapting the proof of \cite[Proposition~2.1]{CHerr}. Here, the condition $\ord\big((f(x),J_f^2(x)\big)\geq m-1$ for $x\in \cO_L^n$ means that we have $\ord\big(g(x)\big)\geq m-1$ for every polynomial $g$ in the ideal of the polynomial ring over $\cO_L$ generated by $f$ and $J_f^2$.  We
 use the orthogonality of characters in the form
\begin{equation}\label{eq:orthchar}
\int_{z \in \cM_L^{m-1}} \psi(z)|dz| = 0
\end{equation}
for $m=m(\psi)$ and combine this with the Taylor expansion, as follows. Let $x_0$ be a point in $\cO_L^m$ such that $\ord f(x_0)\geq m-1$ and suppose that
$\ord\big(J_f^2(x_0)\big) < m-1$. In order to prove (\ref{eq:EfZJ}), it is enough to show that for every such $x_0$, we have
\begin{equation}\label{eq:orth}
\int_{x\in x_0+(\cM_L^{\overline m})^n} \psi\big(f(x)\big) |dx| = 0,
\end{equation}
with $\overline m$ equal to $m/2$ if $m$ is even, and equal to $(m+1)/2$ if $m$ is odd.
If $x=y+x_0$, with $y=(y_1,\ldots,y_n)\in (\cM_L^{\overline m})^n$, then we write the Taylor expansion of $f$ around $x_0$:
$$
f(x) =  f(y+x_0) =   a_0 + \sum_{i=1}^n a_i y_i + \mbox{higher order terms in } y.
$$
Since $p_L$ is assumed to be large and $y\in (\cM_L^{\overline m})^n$, we see  that
$$
f(x) \equiv a_0 + \sum_{i=1}^n a_i y_i \quad (\bmod \cM_L^m)
$$
and thus, since $\psi$ is trivial on $\cM_L^m$, we obtain
\begin{equation}\label{eq:linear}
\int_{x\in x_0+(\cM_L^{\overline m})^n} \psi\big(f(x)\big) |dx| = \int_{y\in (\cM_L^{\overline m})^n} \psi \left(a_0 + \sum_{i=1}^n a_i y_i\right) |dy|.
\end{equation}
Note that the condition $\ord\big(J_f^2(x_0)\big) <m-1$ implies that $\min_i \ord a_i < (m-1)/2$. Since $(m-1)/2 + {\overline m} \leq m$, using the orthogonality relation (\ref{eq:orthchar}), we deduce (\ref{eq:orth}) from (\ref{eq:linear}). This completes the proof of (\ref{eq:EfZJ}).

By (\ref{eq:EfZJ}), using the fact that $| \psi(x) |=1$ for all $x$ in $L$, we get
$$
|E_{f,L,\psi}^{Z} | \leq \Vol\big( \{ x\in \cO_{L}^n\mid \overline x \in Z(k_L),\ \ord(f(x),J_f^2(x))\geq m-1 \}\big)  ,
$$
where the volume is taken with respect to the Haar measure $|dx|$ on $L^n$.
Therefore the existence of $c$ as desired follows from Corollary 2.9 of \cite{VeysZ} and the two sentences following that corollary, which give a link between the log canonical threshold at $Z$ of any ideal ${\mathfrak a}$ of $\cO_L[x]$ and the volume of
 $\{ x\in \cO_{L}^n\mid \overline x\in Z(k_L),\ \ord\big({\mathfrak a}(x)\big)\geq m \}$ uniformly in
 $m>1$.
\end{proof}

\begin{proof}[Proof of Theorem \ref{thm_moi2}]
The formula (\ref{eq:moifJ2}) follows from Proposition \ref{prop:E^ZfJf}
and the definition of $\moi_Z(f)$. Moreover, if $\lct_{Z}(f,J_f^2)\leq 1$, then
part ii) of Theorem~\ref{thmA_intro} 
gives that the hypersurface defined by $f$ does not have
rational singularities in any open neighborhood of $Z$.
Under this last condition, it follows from
\cite[Proposition~3.10]{CMN} that $\moi_Z(f) =\lct_Z(f)$, which together with (\ref{eq:moifJ2})  and $\lct_Z(f) \leq \lct_{Z}(f,J_f^2)$ implies $\moi_Z(f)  = \lct_{Z}(f,J_f^2) = \lct_Z(f)$.   The fact that
when $\lct_{Z}(f,J_f^2)\leq 1$ we also have
$\oi_Z(f) =\lct_Z(f)$  follows by slightly adapting the proof of \cite[Proposition~3.10]{CMN}. 
\end{proof}

\begin{proof}[Proof of Corollary \ref{cor:monodromy}]
Let $L$ and $t_0$ be as in the Corollary. Suppose that the residue field characteristic of $L$ is large enough so that (\ref{eq:EZ}) of Proposition \ref{prop:E^ZfJf} holds with $Z=\A^n$.
By (\ref{eq:EZ}) for this $Z$ and $L$ and Igusa's result from \cite{Igusa3} (in the form of  Corollary 1.4.5  of \cite{Denef} together with the comment at the end of \cite{DenefVeys}), it follows that $t_0$ is equal to $-1$. 
The value $-1$ is automatically a zero of the Bernstein-Sato polynomial of $f$, and hence the corollary is proved.
\end{proof}

\section{Lct-maximal ideals}\label{open_questions}

Let $X$ be a smooth, irreducible, complex algebraic variety of dimension $n$. Recall that all ideals of $\cO_X$ we consider are coherent.

\begin{defi}
A proper nonzero ideal $\fra$ on $X$ is \emph{lct-maximal} if for every ideal $\frb$, with $\fra\subsetneq\frb$, we have
$\lct(\fra)<\lct(\frb)$.
\end{defi}

\begin{rmk}
Because of Noetherianity, for every proper nonzero ideal $\fra$ of $\cO_X$, there is an ideal $\fra'\supseteq \fra$ which is lct-maximal and such that
$\lct(\fra)=\lct(\fra')$.
\end{rmk}

\begin{rmk}
It would be interesting to have a procedure that for a proper nonzero ideal $\fra$ would give an lct-maximal ideal $\fra'$ containing $\fra$ and such that
$\lct(\fra)=\lct(\fra')$. For example, Corollary~\ref{corD_intro} gives that if $\lct(\fra)<1$, then $\lct(\fra)=\lct\big(\fra+D(\fra)^2\big)$. However, we can't iterate this:
the map that takes $\fra$ to $\fra+D(\fra)^2$ is a closure operation (this follows from the fact that $D\big(\fra+D(\fra)^2\big)=D(\fra)$). Note also that
sometimes $\fra+D(\fra)^2$ is not lct-maximal: for example, it follows from Theorem~\ref{thm_char2} that this is the case if $\fra=(f)$ is an ideal defining a hypersurface whose
singular locus (with the reduced scheme structure) doesn't have rational singularities.
\end{rmk}

If $v$ is a divisorial valuation on $X$, then for every positive integer $m$, we consider the coherent ideal $\fra_m(v)=\{f\in\cO_X\mid v(f)\geq m\}$.
Note that the support of the closed subscheme defined by $\fra_m(v)$ is the center of $v$ on $X$.

\begin{defi}
We say that $v$ \emph{computes a log canonical threshold} if it computes $\lct(\fra)$ for some proper nonzero ideal $\fra$.
\end{defi}


\begin{prop}\label{prop_char1}
A proper nonzero ideal
$\fra$ of $\cO_X$ is lct-maximal if and only if for every divisorial valuation $v={\rm ord}_E$ that computes $\lct(\fra)$, we have
$\fra=\fra_m(v)$, where $m=v(\fra)$.
\end{prop}

\begin{proof}
Suppose first that $\fra$ is lct-maximal and
$v={\rm ord}_E$ computes $\lct(\fra)$. Let $m=v(\fra)$, so that $\fra\subseteq \frb:=\fra_m(v)$.
By hypothesis, if $\fra\neq\frb$, then $\lct(\frb)>\lct(\fra)$.
However, we have $v(\frb)=m$, hence
$$\lct(\frb)\leq \frac{A_X(v)}{v(\frb)}=\frac{A_X(v)}{m}=\lct(\fra),$$
a contradiction.

Suppose now that $\fra\subsetneq\frb$ and $\lct(\fra)=\lct(\frb)$. Let $v={\rm ord}_E$ be a divisorial valuation
that computes $\lct(\frb)$ and let $m=v(\fra)$. Since $\fra\subseteq\frb$, we have $v(\frb)\leq m$, hence
$$\lct(\fra)\leq\frac{A_X(v)}{m} \leq\frac{A_X(v)}{v(\frb)}=\lct(\frb).$$
Since $\lct(\fra)=\lct(\frb)$, we conclude that $v(\frb)=m$ and $v$ computes $\lct(\fra)$, but
$\fra\subsetneq \frb\subseteq \fra_m(v)$. This completes the proof of the proposition.
\end{proof}

\begin{question}\label{q1}
If $\fra$ is an lct-maximal ideal in $\cO_X$, can there be two distinct divisorial valuations ${\rm ord}_{E_1}$ and ${\rm ord}_{E_2}$ that compute
$\lct(\fra)$?
\end{question}

\begin{rmk}\label{rmk2_prop_char1}
Note that it follows from Proposition~\ref{prop_char1} that  if the log canonical threshold of an lct-maximal ideal $\fra$ is computed by several
divisorial valuations, then all these have the same center on $X$, namely the support of the subscheme defined by $\fra$.
\end{rmk}

\begin{rmk}\label{rmk_prop_char1}
It follows from Proposition~\ref{prop_char1} that if $\fra$ is an lct-maximal ideal, then $\fra$ is integrally closed and its zero-locus is irreducible.
\end{rmk}

\begin{rmk}\label{local_nature}
The notion of \emph{lct-maximal ideals} is local in the following sense. If $\fra$ is an lct-maximal ideal, then for every open subset $U\subseteq X$
that intersects the zero-locus $V(\fra)$ of $\fra$, the restriction $\fra\vert_U$ is lct-maximal. In the other direction, if $V(\fra)$ is connected and there are open subsets $U_1,\ldots,U_r$
of $X$ with $V(\fra)\subseteq\bigcup_iU_i$ and such that each ideal $\fra\vert_{U_i}$ is lct-maximal, then $\fra$ is lct-maximal. Indeed, both assertions follow easily 
from Proposition~\ref{prop_char1}.
\end{rmk}

\begin{cor}\label{cor_prop_char1}
If $\fra$ is a proper nonzero ideal of $\cO_X$ and $q$ is a positive integer such that the integral closure $\overline{\fra^q}$ of $\fra^q$
is lct-maximal, then $\overline{\fra}$ is lct-maximal.
\end{cor}

\begin{proof}
Note first that a divisorial valuation $v={\rm ord}_E$ computes $\lct(\overline{\fra})=\lct(\fra)$ if and only if it computes
$\lct(\overline{\fra^m})$. If this is the case and $m=v(\fra)=v(\overline{\fra})$, then it follows from the hypothesis, using
Proposition~\ref{prop_char1}, that
$$\overline{\fra^q}=\fra_{mq}(v).$$
Note now that if $f\in \fra_m(v)$, then $f^q\in \fra_{mq}(v)=\overline{\fra^q}$, hence
$f\in\overline{\fra}$. Using again Proposition~\ref{prop_char1}, we conclude that $\overline{\fra}$ is lct-maximal.
\end{proof}

In light of Proposition~\ref{prop_char1}, the following is the main question regarding the description of lct-maximal ideals:

\begin{question}
Given a divisorial valuation $v={\rm ord}_E$ over $X$ that computes a log canonical threshold, when can we find a positive integer $m$ such that $\fra_m(v)$ is lct-maximal? What are these $m$?
\end{question}

\begin{eg}\label{eg1_last}
It follows from Proposition~\ref{rat_lct_maximal} that if $H$ is an irreducible hypersurface in $X$, with rational singularities, then
$\cO_X(-H)$ is lct-maximal. In fact, $\cO_X(-jH)$ is lct-maximal for every $j\geq 1$. Indeed, a valuation $v={\rm ord}_E$
computes ${\rm lct}\big(\cO_X(-jH)\big)$ if and only if it computes ${\rm lct}\big(\cO_X(-H)\big)=1$; the argument in the
proof of
Proposition~\ref{rat_lct_maximal} then shows that this is the case if and only if $E=H$. The assertion then follows immediately from
Proposition~\ref{prop_char1}
\end{eg}

\begin{eg}\label{eg2_last}
It can happen that a divisorial valuation $v={\rm ord}_E$ computes a log canonical threshold, but no ideal $\fra_m(v)$ is lct-maximal.
For example, suppose that $v={\rm ord}_H$, where $H$ is an irreducible divisor on $X$ such that the pair $(X,H)$ is log canonical, but $H$ does not have rational
singularities. The argument in the proof of Proposition~\ref{rat_lct_maximal} shows that in this case there is an exceptional divisor $F$
over $X$ such that $w={\rm ord}_F$ computes $\lct\big(\cO_X(-H)\big)=1$.
This implies that none of the ideals $\fra_m(v)=\cO_X(-mH)$ is lct-maximal (see Remark~\ref{rmk2_prop_char1}).
\end{eg}

We end with the proof of the result stated in Introduction, which extends Examples~\ref{eg1_last} and \ref{eg2_last} to higher codimension:

\begin{proof}[Proof of Theorem~\ref{thm_char2}]
We first prove i). Let $W$ be the support of the closed subscheme defined by $\fra$.
Since $\fra$ is lct-maximal, it follows from Proposition~\ref{prop_char1} that for every divisorial valuation $v={\rm ord}_E$ that
computes $\lct(\fra)$, we have $c_X(v)=W$. In order to show that $W$ has rational singularities, after covering $W$ by affine open subsets of $X$, we see that we may
assume that $X$ is affine (see Remark~\ref{local_nature}). Fix a positive integer $k>\lct(\fra)$ and let $f_1,\ldots,f_k$ be general linear combinations, with coefficients in ${\mathbf C}$, of a system
of generators of $\fra$. If $f=f_1\cdots f_k$, then $\lct(f)=\tfrac{1}{k}\lct(\fra)$ (see for example \cite[Proposition~9.2.28]{Lazarsfeld}). Moreover, the
proof in ${\mathit loc.\,cit.}$ also shows that a divisorial valuation $v={\rm ord}_E$ computes $\lct(\fra)$ if and only if it computes $\lct(f)=c$. We thus conclude that
$W$ is a \emph{minimal log canonical center} for the pair $(X,D)$, where $D=c\cdot {\rm div}(f)$ (this means that it is minimal, with respect to inclusion, among all centers of divisorial valuations
that compute $\lct(D)=1$. In this case, Kawamata's Subadjunction theorem \cite[Theorem~1]{Kawamata} implies that $W$ has rational singularities (for the local case that we need here,
see \cite[Theorem~1.2]{FG}).

Under the assumptions in ii), it is shown in \cite[Theorem~2.1]{Mustata} that if $r={\rm codim}_XW$, then $\lct(\fra)=r$ and the only divisorial valuation
$v={\rm ord}_E$ that computes $\lct(\fra)$ is given by the unique exceptional divisor on the blow-up of $X$ along $W$ that dominates $W$
(recall that since $W$ is Gorenstein, it has rational singularities if and only if it has canonical singularities by a result of Elkik, see \cite[Theorem~11.1]{Kollar}).
The conclusion in ii) thus follows using Proposition~\ref{prop_char1}.
\end{proof}

\end{document}